\UseRawInputEncoding
\documentclass[oneside, 12pt]{amsart}
\usepackage{amsmath}
\usepackage{amssymb}
\usepackage{dsfont}
\usepackage{color}
\usepackage[usenames,dvipsnames,x11names,svgnames]{xcolor}
\usepackage[colorlinks=true,linkcolor=NavyBlue,citecolor=DarkGreen, urlcolor=Blue]{hyperref}
\usepackage{amsthm}
\usepackage{amscd}
\usepackage{geometry}
\usepackage{mathrsfs}
\usepackage{mathtools}

\newtheorem{thm}{Theorem}
\newtheorem{conj}{Conjecture}
\newtheorem{prop}{Proposition}
\newtheorem{lem}{Lemma}

\theoremstyle{definition}

\newtheorem*{rem}{Remark}

\newtheorem*{ack}{Acknowledgements}

\newcommand{\mc}{\mathcal}
\newcommand{\mf}{\mathfrak}

\newcommand{\leqs}{\leqslant}

\newcommand{\geqs}{\geqslant}

\newcommand{\ul}{\underline}

\newcommand{\be}{\begin{equation*}}
\newcommand{\ee}{\end{equation*} }

\newcommand{\ben}{\begin{equation}}
\newcommand{\een}{\end{equation} }

\newcommand{\bs}{\begin{split}}
\newcommand{\es}{\end{split}}

\newcommand{\bmu}{\begin{multline*}}
\newcommand{\emu}{\end{multline*}}

\newcommand{\bmun}{\begin{multline}}
\newcommand{\emun}{\end{multline}}

\begin{document}

\title[The splitting conjecture]{On the splitting conjecture in the hybrid model for the Riemann zeta function}

\author{Winston Heap}
%\address{Max Planck Institute for Mathematics, Vivatsgasse 7, 53111 Bonn.}
\email{winstonheap@gmail.com}

\begin{abstract}We show that the splitting conjecture in the hybrid model of Gonek--Hughes--Keating holds to order on the Riemann hypothesis. Our results are valid in a larger range of the parameter $X$ which mediates between the partial Euler and Hadamard products. We also show that the asymptotic splitting conjecture holds for this larger range of $X$ in the cases of the second and fourth moments. %We utilise the recent developments of Soundararajan, Harper and Radziwi\l\l--Soundararajan on the moments of zeta and $L$-functions.
\end{abstract}

\maketitle

\section{Introduction}

The moments of the Riemann zeta function have been the subject of several conjectural methods in recent years. Since the second and fourth moments of Hardy--Littlewood \cite{HL} and Ingham \cite{I}, it is only relatively recently that a full conjecture for all moments was given. This began with the work of Keating--Snaith \cite{KS} who used the now famous connection with random matrix theory to conjecture that  for real $k>-1/2$,
\[
\frac{1}{T}\int_T^{2T}|\zeta(\tfrac12+it)|^{2k}dt\sim a(k)g(k)(\log T)^{k^2}
\]
where 
\begin{equation}\label{a(k)}
a(k)=\prod_p \bigg(1-\frac1p\bigg)^{k^2}\sum_{m\geqs 0}\frac{d_k(p^m)^2}{p^m}
\end{equation}
and 
\begin{equation}\label{g(k)}
g(k)=\frac{G(k+1)^2}{G(2k+1)}
\end{equation}
where $G$ is the Barnes $G$-function. This was preceded by conjectures for the 6th and 8th moments due to Conrey--Ghosh \cite{CG} and Conrey--Gonek \cite{CGk}, respectively, using  number theoretic methods. The Keating--Snaith conjecture has since been derived with various different approaches \cite{cfkrs,DGH,GHK}. 

A drawback of Keating and Snaith's method was that the arithmetic factor $a(k)$ had to be incorporated in an ad-hoc way since there was no input from primes in their random matrix theory model. This was remedied in the method of Gonek--Hughes--Keating (G--H--K) \cite{GHK} which forms the main focus of this paper.  

The first step of G--H--K's method was to express the zeta function as the product of partial Euler and Hadamard products. Precisely, Theorem 1 of \cite{GHK} states that for $2\leqs X\leqs t^{1/3}$ and large $t$,  
\begin{equation}
\label{hybrid formula}
\zeta(\tfrac{1}{2}+it)=P_X(\tfrac{1}{2}+it)Z_X(\tfrac{1}{2}+it)\Big(1+O\Big(\frac{1}{\log X}\Big)\Big)
\end{equation}
where 
\[
P_X(s)=\exp\bigg(\sum_{n\leqs X}\frac{\Lambda(n)}{n^s\log n}\bigg)
,\,\,\,\,\,\,\,
Z_X(s)=\exp\bigg(-\sum_{\rho}U((s-\rho)\log X)\bigg)
\]
and\[
U(z)=\int_{0}^\infty u(x)E_1(z\log x)dx
\]
with $E_1(z)=\int_z^\infty e^{-w}dw/w$ and $u(x)$ a smooth, non-negative function of mass 1 with support in $[e^{1-1/X},e]$.  
To give a rough idea of these objects, note that from the support conditions on $u$ we have $U(z)\approx E_1(z)$. This has mass concentrated in the region $z\ll 1$ where we have the approximation $E_1(z)\approx-\gamma-\log z$. Thus, roughly speaking, $Z_X(s)\approx \prod_{|\Im(s)-\Im(\rho)|\ll1/\log X} ((s-\rho)e^\gamma\log X)$.  
%\[
%\sum_{n \leqs X}\frac{\Lambda(n)}{n^s\log n}=\sum_{p,m: p^m\leqs X}\frac{1}{mp^{ms}}\approx \sum_{p\leqs X}\log(1-p^{-s})^{-1}
%\] 
Also, from the definition of the von Mangoldt function and the Taylor series for the logarithm we find  $P_X(s)\approx \prod_{p\leqs X}(1-p^{-s})^{-1}$. Therefore, we can indeed view $P_X(s)$ and $Z_X(s)$ as partial Euler and Hadamard products.

G--H--K then proceeded to compute the moments of the Euler product, showing that for $X\ll (\log T)^{2-\epsilon}$,
\begin{equation}\label{P moment}
\frac{1}{T}\int_{T}^{2T}|P_X(\tfrac{1}{2}+it)|^{2k}dt\sim a(k)(e^\gamma\log X)^{k^2},\qquad k\in\mathbb{R}.
\end{equation}
They conjectured with random matrix theory that 
 \begin{equation}\label{Z moment}
\,\,\,\,\frac{1}{T}\int_{T}^{2T}|Z_X(\tfrac{1}{2}+it)|^{2k}dt
\sim
g(k)\Big(\frac{\log T}{e^{\gamma}\log X}\Big)^{k^2},\,\,\,\,\,\,\,\, k>-1/2
\end{equation}
%\[
%\frac{1}{T}\int_{T}^{2T}|Z_X(\tfrac{1}{2}+it)|^{2k}dt
%\sim \begin{cases}
%g(1)\cdot\frac{\log T}{e^{\gamma}\log X},\qquad &k=1
%\\
%g(2)\cdot\Big(\frac{\log T}{e^{\gamma}\log X}\Big)^4,\,\, &k=2
%\end{cases}
%\]
and then proved this in the cases $k=1,2$ for $X\ll (\log T)^{2-\epsilon}$.
% \[
%\frac{1}{T}\int_{T}^{2T}|Z_X(\tfrac{1}{2}+it)|^{2k}dt
%\sim
%g(k)\Big(\frac{\log T}{e^{\gamma}\log X}\Big)^{k^2}
%\]
In order to recover the Keating--Snaith conjecture they assumed that the 
moments of the product of $P_X$ and $Z_X$ should split as the product of moments. 
\begin{conj}[Splitting conjecture, \cite{GHK}] Let $X, T\to\infty$ with $X\ll (\log T)^{2-\epsilon}$. Then for fixed $k> -1/2$ we have 
\[
\frac{1}{T}\int_{T}^{2T}|P_X(\tfrac{1}{2}+it)Z_X(\tfrac{1}{2}+it)|^{2k}dt
\sim 
\frac{1}{T}\int_{T}^{2T}|P_X(\tfrac{1}{2}+it)|^{2k}dt\cdot 
\frac{1}{T}\int_{T}^{2T}|Z_X(\tfrac{1}{2}+it)|^{2k}dt.
\]
\end{conj}
Their  reasoning behind this conjecture was that since $P_X$ and $Z_X$ oscillate at different scales ($1/\log X$  vs.$\,1/\log T$), their contributions should act independently and hence the moment should split to leading order. They verified this in the cases $k=1,2$ for $X\ll (\log T)^{2-\epsilon}$. The methodology of the hybrid model has since been used in various different settings to acquire conjectures for all sorts of $L$-functions \cite{AS,BK,BK2,BGM,BF,D,H}. In all cases, an equivalent version of the splitting conjecture plays a key role.  

In this paper we prove that the splitting conjecture holds to order on the Riemann hypothesis (RH). Furthermore, we can extend the range of $X$ past $(\log T)^{2-\epsilon}$. 

\begin{thm}\label{main thm} Assume RH. Let $\epsilon, k>0$ be fixed and suppose $X, T\to\infty$ with $X\leqs (\log T)^{\theta_k-\epsilon}$ where
$\theta_k=2\sqrt{1+{1}/{2|k|}}$. Then 
\[
\frac{1}{T}\int_{T}^{2T}|P_X(\tfrac{1}{2}+it)Z_X(\tfrac{1}{2}+it)|^{2k}dt
\asymp
\frac{1}{T}\int_{T}^{2T}|P_X(\tfrac{1}{2}+it)|^{2k}dt\cdot 
\frac{1}{T}\int_{T}^{2T}|Z_X(\tfrac{1}{2}+it)|^{2k}dt.
\]
\end{thm}

%\begin{rem}
%Using similar ideas from proof one can improve the range of $X$ significantly to something of the order $T^{o(1)}$ if one restricts the integrals on either side to subsets of `typical' $t\in[T,2T]$. We say more on this in the section \ref{final sec}.
%\end{rem}

As mentioned, this holds in a range of $X$ larger than originally conjectured.  We can also extend the range of $X$ in the asymptotic results \eqref{P moment} and \eqref{Z moment}, both unconditionally and on RH. This gives the following. 

\begin{thm}\label{asymp thm}
The Splitting conjecture holds for $k=1,2$ in the range \[X\leqs \tfrac{1}{10^4}(\log T)^2(\log_2 T)^2.\]
Assuming RH, we may take 
\[
X\leqs 
\begin{cases}
(\log T)^{\sqrt{6}-\epsilon}\,\,\text{ when } &k=1,
\\
(\log T)^{\sqrt{5}-\epsilon}\,\,\text{ when } &k=2.
\end{cases}
\]
%$X\leqs (\log T)^{\sqrt{6}-\epsilon}$ when $k=1$ and $X\leqs (\log T)^{\sqrt{5}-\epsilon}$ when $k=2$.
\end{thm}

Our proofs utilise the recent developments in the theory of moments of $L$-functions due to Soundararajan \cite{S}, Harper \cite{Ha} and Radziwi\l\l--Soundararajan \cite{RS1}. These techniques were originally geared for upper bounds although they can be brought to bear on lower bounds too \cite{HS}.  We highlight three main ideas. 

The first is an innovation of Soundararajan \cite{S}. This was to note that $\log|\zeta(\tfrac{1}{2}+it)|$ can be bounded from above by a sum over primes alone since the zeros contribute negatively to this quantity (see Lemma \ref{sound lem} below and c.f. $\!\!$formula \eqref{hybrid formula}). With this, $|\zeta(\tfrac{1}{2}+it)|$ can be bounded from above by an Euler product of flexible length.

The second idea can be found in a paper of Radziwi\l\l \,\,\cite{R} and features heavily in the later works of Harper \cite{Ha} and Radziwi\l\l--Soundararajan \cite{RS1}. It allows one to compute moments of Euler products provided one can restrict to a certain subset of $[T,2T]$. For the purposes of this discussion we consider the example 
\[
\exp\Big(\sum_{p\leqs Y}p^{-1/2-it}\Big)
\]  
 with $Y=T^{1/(\log\log T)^2}$. On the face of it, this is a very long Dirichlet polynomial. However, if we can restrict $t$ to a subset of $[T,2T]$ on which $|\sum_{p\leqs Y}p^{-1/2-it}|\leqs V$ for a given  $V$, then we can truncate the exponential series effectively using the fact that
 \begin{equation}\label{basic exp trunc}
 e^z\sim \sum_{j=0}^{10V}\frac{z^j}{j!}
 \end{equation}
 for $|z|\leqs V$ and large $V$. The choice of $V$ is naturally dictated by the variance: setting $V=\sum_{p\leqs Y}p^{-1}\sim \log\log T$ we get a Dirichlet polynomial of length $Y^{10V}=T^{10/\log\log T}$. This is now short and so the mean square is easily computed. Also, the exceptional set in this case is of small measure. 
 
 The final main input in the arguments of Harper and Radziwi\l\l--Soundararajan allows one to push the length of the prime sum up to $Y=T^\theta$, for some fixed $\theta>0$. This involves breaking the sum it into subsums of progressively smaller variance. A similar splitting has appeared in the work of Brun on the pure sieve (see Hooley's refinement \cite{Hoo}).  

This circle of ideas has been used in a wide variety of different contexts recently. These include; short interval maxima of the Riemann zeta function \cite{ABBRS, ABR, AOR}, unconditional bounds for the moments of zeta and $L$-functions \cite{Gao, HRS, HS}, value distribution of $L$-functions \cite{Das, HW,RSsel}, sign changes in Fourier coefficients of modular forms \cite{LR},  non-vanishing of central values of  $L$-functions \cite{DFL}  and equidistribution of lattice points on the sphere \cite{HR}. In our case, we use these ideas to prove the following.

\begin{prop}\label{P prop}Let $\epsilon>0$ and $k\in\mathbb{R}$ be fixed. Suppose $X\leqs \eta_k(\log T)^2(\log_2 T)^2$ with $\eta_k=\tfrac{1}{16k^2}-\epsilon$. Then
\[
\frac1T\int_T^{2T}|P_X(\tfrac12+it)|^{2k}dt\sim a(k) (e^\gamma \log X)^{k^2}
\]
where $a(k)$ is given by \eqref{a(k)}. Assuming RH, this holds for $X\leqs (\log T)^{\theta_k-\epsilon}$ with 
$\theta_k=2\sqrt{1+{1}/{2|k|}}.$
\end{prop}

\begin{prop}\label{Z asymp prop}
Suppose $X\leqs \tfrac{1}{10^4}(\log T)^2(\log_2 T)^2$. Then for $k=1,2$ we have  
\[
\frac{1}{T}\int_{T}^{2T}|Z_X(\tfrac{1}{2}+it)|^{2k}dt\sim g(k)\bigg(\frac{\log T}{e^\gamma\log X}\bigg)^{k^2}
\]
where $g(k)$ is given by \eqref{g(k)}. Assuming RH we may take $X\leqs (\log T)^{\sqrt{6}-\epsilon}$ when $k=1$ and $X\leqs (\log T)^{\sqrt{5}-\epsilon}$ when $k=2$.
\end{prop}

\begin{prop}\label{Z prop}
Assume RH and let $\epsilon,k>0$ be fixed. Suppose $X\leqs (\log T)^{\theta_k-\epsilon}$ with $\theta_k$ as above. Then 
\[
\frac{1}{T}\int_{T}^{2T}|Z_X(\tfrac{1}{2}+it)|^{2k}dt\asymp \bigg(\frac{\log T}{\log X}\bigg)^{k^2}.
\]
\end{prop}
\begin{rem}
The lower bound in Proposition \ref{Z prop} can be made unconditional provided $X\leqs\eta_k(\log T)^2(\log_2 T)^2 $. We say more on this in section \ref{lower 1}.
\end{rem}

Since $\int_T^{2T}|\zeta(\tfrac{1}{2}+it)|^{2k}dt$ is $\ll T(\log T)^{k^2}$ on RH \cite{Ha} and $\gg T(\log T)^{k^2}$ unconditionally \cite{HS}, 
Theorem \ref{main thm} follows from Propositions \ref{P prop} and \ref{Z prop} when combined with \eqref{hybrid formula}. Likewise, Theorem \ref{asymp thm} follows on combining Propositions \ref{P prop} and \ref{Z asymp prop} and \eqref{hybrid formula}.

Using the case of $P_X$ as an example, we describe how the range of $X$ can be increased past $(\log T)^{2-\epsilon}$. First of all, note that since 
\begin{equation}\label{sum bound}
\Big|\sum_{n\leqs X}\frac{\Lambda(n)}{n^{1/2+it}\log n}\Big|\leqs (1+o(1))\frac{2X^{1/2}}{\log X},
\end{equation}
we can approximate $P_X(1/2+it)^k$ with a Dirichlet polynomial of length $X^{20|k|X^{1/2}/\log X}$ by using \eqref{basic exp trunc} to truncate the exponential.
If $X\ll(\log T)^{2-\epsilon}$ then this is $T^{o(1)}$ and so we have a short Dirichlet polynomial. Note this holds for \emph{all} $t\in[T,2T]$ since the bound \eqref{sum bound} is pointwise. G--H--K computed a Dirichlet polynomial approximation in a slightly different way, although in order for it to be short they required the same bound on  $X$, perhaps unsurprisingly.

If $X$ is larger, then in order to have a short Dirichlet polynomial we must restrict to a subset of $[T,2T]$ and in this case we need good bounds on the exceptional set. Typically, one would expect Gaussian bounds of the shape
\begin{equation}\label{gauss}
\frac{1}{T}\mu\Big(\Big\{t\in[T,2T]:\Big|\Re\sum_{n\leqs X}\frac{\Lambda(n)}{n^{1/2+it}\log n}\Big|\geqs V\Big\}\Big)\ll \exp\Big(-\frac{V^2}{\log\log T}\Big)
\end{equation}
in a wide range of $V$. In practice we are limited to $V\ll \sqrt{(\log T)(\log_2 T)/\log X}$ which may be much smaller than the maximum $2X^{1/2}/\log X$. For the remaining range of $V$ one must settle for weaker bounds. For example, in \cite{S} it is shown that the tails of $\log|\zeta(1/2+it)|$ can be bounded  by $e^{-V\log V}$ when $V\gg \log_2T \log_3 T$. We can show that the tails of our sum satisfy the same bound in the range $\log_2 T\log_3 T\leqs V\leqs 2X^{1/2}/\log X$ provided $X\ll (\log T)^2$. However, for our purposes the weaker bound of $e^{-AV}$ with large $A$ is sufficient and this affords us slightly more room in the size of $X$.  

Another avenue for improvement is to reduce the trivial bound in \eqref{sum bound}. This becomes a manageable task under RH and thus we are able to make further gains in the size of $X$ under this assumption. We shall prove the following. 

\begin{thm}\label{prime sum thm}Assume RH. Then for large $t\in [T,2T]$ and $ 2(\log T)^2\leqs X\leqs T$ we have 
\[
\Big|\Re\sum_{n\leqs X}\frac{\Lambda(n)}{n^{1/2+it}\log n}\Big|
\leqs 
(\tfrac{1}{2}+o(1))\big(\log(\tfrac{X^{1/2}}{\log T})+4\log\log X\big)\frac{\log T}{\log\log T}.
\]
For the imaginary part we can replace the factor $\tfrac{1}{2}+o(1)$ by $\tfrac{1}{\pi}+o(1)$.
%and
%\[
%\Big|\Im\sum_{n\leqs X}\frac{\Lambda(n)}{n^{1/2+it}\log n}\Big|
%\leqs 
%(\tfrac{1}{\pi}+o(1))\big(\log(\tfrac{X^{1/2}}{\log T})+4\log\log X\big)\frac{\log T}{\log\log T}
%\]
\end{thm}

The factor of $1/2+o(1)$ here is related to the function $S(t)$ and can be read as $2c$ where $c$ is a permissible constant in the bound $S(t)\leqs (c+o(1))\log t/\log\log t$. The current best is due to Carneiro--Chandee--Milinovich \cite{CCM} who give $c=1/4$. Our approach to Theorem \ref{prime sum thm} is to relate the sum with $S(t)$ via contour integrals and then input these bounds.  We have not made attempts to further optimise this argument but it would be interesting to see if one could  use the extremal function machinery of Carneiro et al.$\!$ in a more direct way.  

Regarding further improvements in the size of $X$, if the conjectural bound $S(t)\ll \sqrt{\log t\log_2 t}$ of Farmer--Gonek--Hughes \cite{FGH} holds, then one could take $X\leqs \exp($ $C(\log t)^{1/4})$. Also, assuming that the bounds for the exceptional set in \eqref{gauss}, or some minor variant of this\footnote{In fact, anything of the form $e^{-AV}$ with large $A$ would be sufficient.}, hold in the full range of $V$ for a given $X$, then our arguments can reproduce Theorems \ref{main thm} and \ref{asymp thm} for $X$ as large as $T^{1/C\log\log T}$. This supports the view of G--H--K that the splitting conjecture may hold as long as $X=o(T)$.

The paper is organised as follows. We first prove Theorem \ref{prime sum thm} in section \ref{prime sec} and then the asymptotic results of Propositions \ref{P prop} and \ref{Z asymp prop} in sections \ref{Euler sec} and \ref{Hadamard sec}, respectively. In section \ref{tools sec} we describe some tools for later use. Then in section \ref{upper sec} we prove the upper bound of Proposition \ref{Z prop} and in sections \ref{lower 1} and \ref{lower 2} we prove the lower bound in the cases $0\leqs k\leqs 1$ and $k\geqs 1$, respectively. 

\begin{ack}The author would like to thank Jing Zhao and Junxian Li for their comments on an early draft of this paper and Chris Hughes for some clarifying remarks.
\end{ack}

%%%%%%%%%%%%%%%      PRIME SUM        %%%%%%%%%%%%%%%

\section{Bounds for prime sums: Proof of Theorem \ref{prime sum thm}}\label{prime sec}

We first give a lemma which relates our prime sum to $S(t)$. This shares some similarities with previous convolution formulas from the literature \cite{Sel, Tsang}.

\begin{lem}\label{St lem}
Assume RH. For large $t\in[T,2T]$ and $2\leqs X\leqs T$, $Y\leqs T/2$, we have
\[
\sum_{n\leqs X}\frac{\Lambda(n)}{n^{1/2+it}\log n}
=
\int_{t-Y}^{t+Y}S(y)\frac{1-X^{-i(t-y)}}{t-y}dy+E(X,Y,T)
\]
where 
\begin{equation*}
E(X,Y,T)
\ll
\frac{X^{1/2}(\log X+\tfrac{\log_3 T}{\log X})}{Y}
+
\frac{\log T}{\log_2 T}\Big(\frac{1}{Y^{1/2}}
+
\frac{X^{{C}/{\log_2 T}}}{Y}\mathds{1}_{X>(\log T)^A}\Big)
\end{equation*}
and $A$ is a large constant.
\end{lem}
\begin{proof}
By Perron's formula (Lemma 3.19, \cite{T}) we have 
\[
\sum_{n\leqs X}\frac{\Lambda(n)}{n^{1/2+it}\log n}
=
\frac{1}{2\pi i}\int_{1/2+1/\log X-iY}^{1/2+1/\log X+iY}\log \zeta(z+\tfrac{1}{2}+it)X^z\frac{dz}{z}+O\Big(\frac{X^{1/2}\log X}{Y}\Big).
\]
We shift the contour to the line with real part $\Re(z)=1/\log X$. Restricting $Y\leqs T/2$ we don't encounter any poles. In the region $1/2+C/\log\log \tau\leqs \sigma\leqs 1$, $\tau\gg 1$, we have  
\[
\log \zeta(\sigma+i\tau)\ll \frac{(\log \tau)^{2-2\sigma}\log_3\tau}{\log_2 \tau}.
\]
This follows from (14.14.5) of \cite{T}, the Phragmen--Lindel\"of principle and the bound $\log \zeta(\sigma+i\tau)\ll \log _3\tau$, $\sigma\geqs 1$ (the latter can be deduced from the proof of Theorem 14.8 of \cite{T}). From Theorem 14.14 (B) of \cite{T} we have 
\[
\log\zeta(\sigma+i\tau)\ll \frac{\log \tau}{\log_2\tau}\log\Big(\frac{2}{(\sigma-1/2)\log_2\tau}\Big),\qquad \tfrac{1}{2}<\sigma\leqs \tfrac12+C/\log_2\tau.
\]
Therefore, the horizontal contours contribute
\begin{equation}\label{horizontal}
\ll
\frac{X^{{C}/{\log_2 T}}\log(\tfrac{\log X}{\log_2 T})\log T}{Y\log_2 T\log X}\mathds{1}_{X>(\log T)^A}
+
\frac{X^{1/2}\log_3 T}{Y\log_2 T\max(\log(\tfrac{X^{1/2}}{(\log T)^2}),1)}
+
\frac{X^{1/2}\log_3 T}{Y\log X}.
\end{equation}

By formula (14.10.5) of \cite{T} (see also (14.12.4) there) we have 
\[
\log\zeta(z+\tfrac{1}{2}+it)
=
i\int_{t/2-Y}^{2t+Y}\frac{S(y)}{z+i(t-y)}dy+O\Big(\frac{\log T}{T}\Big).
\]
This implies
\begin{equation}\label{log int}
\frac{1}{2\pi i}\int_{1/\log X-iY}^{1/\log X+iY}\log\zeta(z+\tfrac{1}{2}+it)X^z\frac{dz}{z}
=
i\int_{t/2-Y}^{2t+Y}S(y)I(t-y)dy+O\Big(\frac{(\log T)^2}{T}\Big).
\end{equation}
where 
\[
I(t-y)
=
\frac{1}{2\pi i}\int_{1/\log X-iY}^{1/\log X+iY}\frac{X^z}{z(z+i(t-y))}dz.
\]
A trivial estimate gives 
\begin{equation}\label{I triv} 
I(t-y)\ll \int_{-Y}^{Y}\frac{dx}{(1/\log X+|x|)(1/\log X+|x+t-y|)}\ll \log X.
\end{equation}
On the other hand, shifting the contour to the left we find 
\begin{equation}\label{I res}
I(t-y)
=
\frac{1-X^{-i(t-y)}}{i(t-y)}\mathds{1}_{|t-y|\leqs Y}
+
\frac{1}{i(t-y)}\mathds{1}_{|t-y|> Y}
+
O\Big(\frac{1}{Y|Y-|t-y||\log X}\Big).
\end{equation}

For a given $\delta>0$ to be chosen, we find by \eqref{I triv} that
\[
\int_{|t-y\pm Y|\leqs \delta}S(y)I(t-y)dy\ll \delta \log X\log T/\log_2 T
\]
since $S(\tau)\ll \log \tau/\log_2 \tau$. In the integral over the remaining region, the error term of \eqref{I res} contributes
\[
\ll \frac{\log T}{\delta Y\log X\log_2 T}+\frac{\log Y\log T}{Y\log X\log_2 T}
\]
after considering the regions $\delta<|Y-|t-y||\leqs 1$ and $1<|Y-|t-y||$ separately. Therefore, on choosing $\delta=1/(Y^{1/2}\log X)$ we find that the integral on the right of \eqref{log int}
is 
\begin{multline*}
i\int_{\substack{t/2-Y\leqs y\leqs 2t+Y\\ |t-y\pm Y|> \delta}}S(y)
\bigg(
\frac{1-X^{-i(t-y)}}{i(t-y)}\mathds{1}_{|t-y|\leqs Y}
+
\frac{1}{i(t-y)}\mathds{1}_{|t-y|> Y}
\bigg)dy
+
O\Big(
\frac{\log T}{Y^{1/2}\log_2 T}\Big).
\end{multline*}
Integrating by parts along with the bound $S_1(\tau)\ll \log \tau/(\log_2 \tau)^2$, we find that the second term in this integral is $\ll \log T/(Y(\log_2 T)^2)$ which can be absorbed into the error term immediately above. The range of integration of the first term can be extended to $|t-y|\leqs Y$ at the cost of an error $\ll \log T/(Y\log_2T)$. Combining this in \eqref{log int} along with the error terms of \eqref{horizontal} the result follows.

\end{proof}

\begin{proof}[Proof of Theorem \ref{prime sum thm}]
We apply Lemma \ref{St lem} with 
\[
Y=\frac{X^{1/2}(\log X)^{2+\epsilon}}{\log T}.
\]
With this choice we have $E(X,Y,T)=o(\log T/\log\log T)$ since $\log X\gg\log_2 T$. Therefore, on taking real parts in Lemma \ref{St lem} we find
\[
\Re\sum_{n\leqs X}\frac{\Lambda(n)}{n^{1/2+it}\log n}
=
-\int_{-Y}^{Y}S(t+y)\frac{1-\cos (y\log X)}{y}dy+o(\log T/\log_2 T).
\]
From \cite{CCM} we have
\[
|S(\tau)|\leqs (1+o(1))\frac{1}{4}\frac{\log\tau}{\log\log\tau}
\]
 for large $\tau$ and so
 \[
 \Big|
 \Re\sum_{n\leqs X}\frac{\Lambda(n)}{n^{1/2+it}\log n}
\Big|
\leqs (1+o(1))\frac{\log T}{4\log\log T}\int_{-Y\log X}^{Y\log X} \frac{1-\cos y}{|y|}dy.
 \]
The integral here is 
\begin{align*}
2\int_{\pi}^{Y\log X}  \frac{1-\cos y}{y}dy+O(1)
\leqs  &
2\sum_{n=1}^{Y\log X}\frac{1}{\pi n}\int_{\pi n}^{\pi(n+1)}(1-\cos y)dy+O(1)
\\
= &
2 \log(Y\log X)+O(1) .
\end{align*}
Thus, we acquire 
\[
 \Big|
 \Re\sum_{n\leqs X}\frac{\Lambda(n)}{n^{1/2+it}\log n}
\Big|
\leqs
(1+o(1))\frac{\log(Y\log X)\log T}{2\log\log T}
\]
and  the result follows on inputting our choice of $Y$. 

In the case of the imaginary part, we acquire the integral $\int_{-Y}^Y|\sin(y\log X)/y|dy$ which, on following the same argument, is $\leqs (4/\pi)\log(Y\log X)+O(1)$.
\end{proof}

From the proof we see that the factor $\log(X^{1/2}/\log T)$ comes from the divergent integral $\int_{-Y}^Y|1-X^{-iy}|/|y|dy$. One may then wonder if smoothing would help here, that is, if the problem could be modified so that we consider a smoothed sum instead; say 
\[
\sum_{n\leqs X}\frac{\Lambda(n)}{n^{1/2+it}\log n}\Big(1-\frac{\log n}{\log X}\Big).
\] 
The imaginary part responds well to this procedure and the above proof recovers the formulas of Selberg \cite{Sel} and Tsang \cite{Tsang} which have the convergent integrand  $\sin^2(y\log X)/y^2$. Unfortunately, for the real part the integrand is again of the form $\approx 1/|y|$ owing to large negative values of $\Re\log\zeta(1/2+it)$ (cf. Lemma 5 of \cite{Tsang}).

%%%%%%%%%%%%%%%%%%% EULER PROD %%%%%%%%%%%%%

\section{Moments of the Euler product: Proof of Proposition \ref{P prop}}\label{Euler sec}

Recall that  
\[
P_X(s)=\exp\Big(\sum_{n\leqs X}\frac{\Lambda(n)}{n^s\log n}\Big).
\]
Here and throughout the paper we consider the following subsets of $[T,2T]$ on which the sum in the exponential attains typical values. For $V\geqs 0$ set
\begin{equation*}\label{S}
\qquad\qquad\mc{S}_{\Re}(V)=\Big\{t\in[T,2T]:\Big|\Re\sum_{n\leqs X}\frac{\Lambda(n)}{n^{1/2+it}\log n}\Big|\leqs V\Big\},  
\end{equation*}
and define the set relating to the imaginary part $\mc{S}_{\Im}(V)$ similarly. Let
\[
\mc{S}(V)
=
\mc{S}_{\Re}(V)\cap \mc{S}_{\Im}(V)
\] 
and denote
\[
V_0=\log_2 T\log_3 T,\qquad\,\,\,\,\,\,\mc{S}=\mc{S}(V_0).
\]
Define the complementary sets by 
\[
\mc{E}_{\Re}(V)=\overline{\mc{S}_{\Re}(V)},\,\,
\mc{E}_{\Im}(V)=\overline{\mc{S}_{\Im}(V)},\,\,
\mc{E}(V)=\overline{\mc{S}(V)},\,\,
\mc{E}=\overline{\mc{S}}\,\,
\]
where $\overline{A}=[T,2T]\backslash A$ for a given set $A$. 
Also, let 
\[
V_{\max}
=
\max_{t\in[T,2T]}\bigg(\Big|\Re\sum_{n\leqs X}\frac{\Lambda(n)}{n^{1/2+it}\log n}\Big|,
\Big|\Im\sum_{n\leqs X}\frac{\Lambda(n)}{n^{1/2+it}\log n}\Big|\bigg).
\]  
Note that unconditionally $V_{\max}\leqs (2+o(1))X^{1/2}/\log X$ and that on the Riemann hypothesis  $V_{\max}\leqs (1/2+o(1))(\log(X^{1/2}/\log T)+O(\log_2 X))\log T/\log_2 T$ by Theorem \ref{prime sum thm}. The reason we work with the real and imaginary parts (as opposed to working with the modulus directly) is so that we have slightly better  conditional bounds for $V_{\max}$ whilst still maintaining control over the modulus (which is important for Lemma \ref{P lem} below). 
This gives better exponents for our logarithms in the conditional results but entails slightly more work.

To the estimate the measure of the complementary sets we use the following.  
\begin{lem}[\cite{S}]\label{sound lem 2}Let $T$ be large and let $2\leqs x\leqs T$. Let $m$ be a natural number such that $x^{m}\leqs T/\log T$. Then for any complex numbers $a(p)$ we have 
\[
\frac{1}{T}\int_{T}^{2T}\bigg|\sum_{p\leqs x}\frac{a(p)}{p^{1/2+it}}\bigg|^{2m}dt\ll m!\bigg(\sum_{p\leqs x}\frac{|a(p)|^2}{p}\bigg)^m.
\]
\end{lem}

\begin{lem}\label{E meas lem} Let $\epsilon,\kappa>0$. Then for $X\leqs \eta_\kappa(\log T)^2(\log_2 T)^2$ with $\eta_\kappa=\tfrac{1}{16\kappa^2}-\epsilon$  we have 
\[
\qquad\qquad\qquad \qquad\qquad \qquad   \mu(\mc{E}_{\Re}(V))
\ll
 Te^{-(2\kappa+o(1)) V},
 \qquad\qquad  V_0\leqs V\leqs V_{\max}
\]
where $\mu$ denotes Lebesgue measure.
Assuming RH, the same bound holds provided $X\leqs (\log T)^{\theta_\kappa-\epsilon}$ where 
\begin{equation}\label{theta}
\theta_\kappa=2\sqrt{1+\tfrac{1}{2\kappa}}.
\end{equation} 
The same results hold for $\mc{E}_{\Im}(V)$ also.
\end{lem}

\begin{proof}
We first prove the unconditional result. Write
\[
\sum_{n\leqs X}\frac{\Lambda(n)}{n^{1/2+it}\log n}
=
\sum_{p\leqs X}\frac{1}{p^{1/2+it}}
+
\sum_{p\leqs \sqrt{X}}\frac{1}{2p^{1+2it}}
+
O(1).
\]
Then from Jensen's inequality in the form $(a+b+c)^{2m}\leqs 3^{2m-1}(a^{2m}+b^{2m}+c^{2m})$ with $m\geqs 1$, we have  
\[
\mu(\mc{E}_{\Re}(V))\leqs \frac{3^{2m-1}}{V^{2m}}\bigg(\int_T^{2T}\Big|\sum_{p\leqs X}\frac{1}{p^{1/2+it}}\Big|^{2m}dt
+
\int_T^{2T}\Big|\sum_{p\leqs \sqrt{X}}\frac{1}{2p^{1+2it}}\Big|^{2m}dt
+O(TC^{2m})
\bigg).
\]
By Lemma \ref{sound lem 2} the right hand side is 
\begin{equation}\label{mth moment bound}
\ll T\frac{3^{2m}m!}{V^{2m}}\Big(\sum_{p\leqs X}\frac{1}{p}\Big)^{m}
\ll
Tm^{1/2}\Big(\frac{9m\log\log X}{eV^2}\Big)^m
\end{equation}
provided $m\leqs (\log T-\log_2 T)/\log X$. Choosing 
\[
m=\frac{2\kappa V}{\log V}
\]
the bound $\ll Te^{-(2\kappa+o(1)) V}$ follows since $\log\log X\leqs V^{o(1)}$. Note that our choice of $m$ is legal since 
 \[
 m\leqs \frac{2\kappa V_{\max}}{\log V_{\max}}
\lesssim
 \frac{8\kappa X^{1/2}}{(\log X)^2}
 \lesssim
 \frac{4\kappa\eta_k^{1/2}\log T}{\log X}
 \leqs
 (1-\epsilon)\frac{\log T}{\log X}
 \]
 where $\lesssim$ means $\leqs$ times a constant of the form $1+o(1)$. 
 
 Assuming RH, then on writing $\theta_\kappa=2+\nu_\kappa$ we find 
 \[
 V_{\max}\leqs 
 (\tfrac{1}{2}+o(1))\big(\log(\tfrac{X^{1/2}}{\log T})+4\log\log X\big)\frac{\log T}{\log\log T}\lesssim \frac{\nu_\kappa-\epsilon}{4}\log T
 \]
 by Theorem \ref{prime sum thm}.
 Choosing $m$ as before gives the desired bound for $\mu(\mc{E}_{\Re}(V))$ and, again, our choice of $m$ is legal since in this case we have 
 \[
 m
 \lesssim
 \frac{\kappa(\nu_\kappa-\epsilon)\log T}{2\log\log T} 
 \]
 which is $\leqs (\log T-\log_2 T)/\log X$ provided 
 \[
 2+\nu_\kappa\leqs \frac{2}{\kappa\nu_\kappa}
 \,\,\iff\,\,
 \nu_\kappa\leqs2\sqrt{1+\tfrac{1}{2\kappa}}-2.
 \]
\end{proof}

\begin{lem}\label{PE lem} Let $\epsilon>0$, $\nu\in\mathbb{R}$ and suppose $X\leqs \eta_\nu(\log T)^2(\log_2 T)^2$ with $\eta_\nu=\tfrac{1}{16\nu^2}-\epsilon$. Then
\[
\frac{1}{T}\int_{\mc{E}}|P_X(\tfrac12+it)|^{2\nu}dt\ll e^{-\delta V_0}
\]
for some $\delta>0$ dependent on $\epsilon$. Assuming RH, this holds provided $X\leqs (\log T)^{\theta_\nu-\epsilon}$ where $\theta_\nu$ is given by \eqref{theta}. 
\end{lem}
\begin{proof}
Since $A\cup B=A\cup (\overline{A}\cap B)$ we may write 
\begin{equation}\label{E decomp}
\mc{E}=\mc{E}_{\Re}(V_0)\cup \mc{E}_{\Im}(V_0)=\mc{E}_{\Re}(V_0)\cup(\mc{S}_{\Re}(V_0)\cap\mc{E}_{\Im}(V_0)).
\end{equation}
The integral over $\mc{E}_{\Re}(V_0)$ is 
\begin{equation}\label{int over e}
\begin{split}
\leqs & \frac{1}{T}\int_{\mc{E}_{\Re}(V_0)}\exp\Big(2|\nu|\cdot \Big|\Re\sum_{n\leqs X}\frac{\Lambda(n)}{n^{1/2+it}\log n}\Big|\Big)dt
\\
= &
\frac1T e^{2|\nu|V_0}\mu(\mc{E}_{\Re}(V_0))
+
\frac{2|\nu|}{T}\int_{V_0}^{V_{\max}} e^{2|\nu|V}\mu(\mc{E}_{\Re}(V))dV.
\end{split}
\end{equation}
Since $\eta_\nu\leqs 1/16(|\nu|+\delta/2)^2-\epsilon/2$ for some $\delta>0$ dependent on $\epsilon$, Lemma \ref{E meas lem} gives $\mu(\mc{E}_{\Re}(V))\ll Te^{-(2|\nu|+\delta)V}$. Hence the quantity in \eqref{int over e} is $\ll e^{-\delta V_0}$. Similarly, the integral over $\mc{S}_{\Re}(V_0)\cap\mc{E}_{\Im}(V_0)$ is
\[
\leqs \frac{1}{T}e^{2|\nu|V_0} \mu(\mc{E}_{\Im}(V_0) )
\leqs e^{-\delta V_0}.
\]
The result then follows by the union bound. Likewise, $\theta_\nu-\epsilon\leqs \theta_{\nu+\delta/2}-\epsilon/2$ and so the same argument gives the conditional case.
\end{proof}

We now show that $P_X(1/2+it)^k$ can be approximated by a short Dirichlet polynomial provided $t\in\mc{S}$. To do this we note that for $|z|\leqs Z$ we have 
\begin{equation}\label{exp trunc 0}
\Big|e^z-\sum_{j=0}^{10Z}\frac{z^j}{j!}\Big|\leqs e^{-10Z}
\end{equation}
by a Taylor expansion and Stirling's formula.

\begin{lem}\label{P lem}
%Let $X\leqs c(\log T\log_2 T/\log_3 T)^2$ and 
Suppose $t\in\mc{S}$. Let $k\in\mathbb{R}$ and define
\[
W_0=W_0(k,T)=20|k|V_0.
\]
%\max\Big(\frac{2X^{1/2}}{\log X},\frac{\log T}{1500X^{1/2}\log X}\Big)$. 
Then as $X\to\infty$,
\begin{equation}\label{P}
\begin{split}
P_X(\tfrac{1}{2}+it)^{k}
= &
\,\,
\big(1+O(e^{-19|k|V_0})\big)D(t,k)
\end{split}
\end{equation}
where
\begin{equation}
\label{D}
D(t,k)
=
\sum_{\substack{n\in S(X)}}\frac{\alpha_k(n)}{n^{1/2+it}}
\end{equation}
with $S(X)=\big\{n\in\mathbb{N}: p|n\implies p\leqs X\big\}$ and where the coefficients $\alpha_k(n)$ satisfy the following properties:
\begin{itemize}
\item $\alpha_k(n)$ is supported on integers $n\leqs X^{W_0}$ and  
\[
|\alpha_k(n)|\leqs d_{|k|}(n).
\]
\item If $\,\,\Omega(n)\leqs W_0$ then $\alpha_k(n)=\beta_k(n)$ where $\beta_k(n)$ is a multiplicative function satisfying  
\[
\beta_k(n)=d_k(n) 
\]
if $p^m|n\implies p^m\leqs X$ and $|\beta_k(n)|\leqs d_{|k|}(n)$ in general.
\end{itemize}
\end{lem}
\begin{proof}
Since $t\in\mc{S}$ we have
\[
\Big|\sum_{\ell\leqs X}\frac{\Lambda(\ell)}{\ell^{1/2+it}\log \ell}\Big|\leqs 2\log_2 T\log_3 T,
\]
and so by \eqref{exp trunc 0} we acquire 
\begin{equation}\label{D coeffs}
P_X(\tfrac12+it)^{k}
= 
\big(1+O(e^{-19|k|V_0})\big)
\sum_{j=0}^{W_0}\frac{k^j}{j!}\bigg(\sum_{\ell\leqs X}\frac{\Lambda(\ell)}{\ell^{1/2+it}\log \ell}\bigg)^j.
\end{equation}
Writing the sum on the right as the Dirichlet polynomial $D(t,k)$ it remains to deduce the properties of the coefficients $\alpha_k(n)$. 

Clearly, this is a Dirichlet polynomial of length $X^{W_0}$ over the $X$-smooth numbers $S(X)$. For the remaining properties, first note that we may write
\begin{equation}\label{prodprod}
\exp\Big(k\sum_{\ell\leqs X}\frac{\Lambda(\ell)}{\ell^{s}\log \ell}\Big)
= 
\prod_{p\leqs X}\exp\Big(-k\log(1-p^{-s})-k\sum_{m: \,\,p^m>X}\frac{1}{mp^{ms}}\Big).
\end{equation}
Note also that after performing a Taylor expansion of the left hand side and collecting like terms for $n^s$, the coefficients are a sum of positive terms if $k>0$, whilst they can be bounded from above by the same sum but involving $|k|$ if $k<0$. From these observations it is clear that $|\alpha_k(n)|\leqs d_{|k|}(n)$ since the right hand side is the generating function for the divisor functions $d_k(n)$ with some terms removed. 

 Moreover, if we form the product on the right hand side of \eqref{prodprod} into a series $\sum_{n\in S(X)}\beta_k(n)n^{-s}$, then we see that the coefficients are multiplicative and satisfy $\beta_k(n)=d_k(n)$ if $p^m|n\implies p^m\leqs X$. Since the highest power $j$ in \eqref{D coeffs} is $W_0$ we see that, certainly, if $\Omega(n)\leqs W_0$ then $\alpha_k(n)=\beta_k(n)$. 
\end{proof}

\begin{proof}[Proof of Proposition \ref{P prop}]Write 
\[
\frac1T\int_T^{2T}|P_X(\tfrac12+it)|^{2k}dt
=
\frac1T\int_\mc{S}|P_X(\tfrac12+it)|^{2k}dt
+
\frac1T\int_\mc{E}|P_X(\tfrac12+it)|^{2k}dt.
\]
By Lemma \ref{PE lem}, the second integral here is $o(1)$.
For the first integral, Lemma \ref{P lem} gives
\begin{align*}
\frac1T\int_\mc{S}|P_X(\tfrac12+it)|^{2k}dt
\sim &
\frac1T\int_\mc{S}|D(t,k)|^{2}dt
\\
= &
\frac1T\int_T^{2T}|D(t,k)|^{2}dt+O\bigg(\frac1T\int_\mc{E}|D(t,k)|^{2}dt\bigg)
\end{align*}
By the Cauchy--Schwarz inequality the error term here is
\begin{equation}\label{D E int} 
\ll \Big(\frac{\mu(\mc{E})}{T}\Big)^{1/2}\bigg(\frac1T\int_T^{2T}|D(t,k)|^{4}dt\bigg)^{1/2}.
\end{equation}
Since $D(t,k)$ is of length $X^{W_0}\leqs e^{C(\log_2 T)^2\log_3 T}$, the Montgomery--Vaughan mean value Theorem (see \eqref{MVT} below) and the coefficient bounds of Lemma \ref{P lem} give 
\begin{align}\label{D fourth}
\frac1T\int_T^{2T}|D(t,k)|^{4}dt
= & 
\,(1+o(1)) 
\sum_{\substack{n_1n_2=n_3n_4\\  n_j\in S(X)}}
\frac{\alpha_k(n_1)\alpha_k(n_2)\alpha_k(n_3)\alpha_k(n_4)}{(n_1n_2n_3n_4)^{1/2}}\nonumber
\\
\ll &
\prod_{p\leqs X}\bigg(1+\frac{4k^2}{p}+O(p^{-2})\bigg)
\ll
(\log X)^{4k^2}.
\end{align}
More generally, for fixed $m\in\mathbb{N}$ we have
\begin{equation}\label{D mth moment}
\frac1T\int_T^{2T}|D(t,k)|^{2m}dt\ll (\log X)^{m^2k^2}.
\end{equation}
Therefore, by Lemma \ref{E meas lem} the expression in \eqref{D E int} is 
\[
\ll e^{-|k|(\log_2 T)(\log_3 T)} (\log X)^{2k^2}=o(1)
\]
and hence we may concentrate on the integral of $|D(t,k)|^2$ over the full set $[T,2T]$.

Applying the Montgomery--Vaughan mean value theorem again gives 
\[
\frac1T\int_T^{2T}|D(t,k)|^{2}dt
=
(1+o(1))\sum_{n\in S(X)}\frac{\alpha_k(n)^2}{n}.
\]
Since $|\alpha_k(n)|\leqs d_{|k|}(n)$, the sum over terms with $\Omega(n)> W_0$ is, for any $1<r<2$,
\begin{equation}\label{tail sum bound}
\sum_{\substack{n\in S(X)\\\Omega(n)>W_0}}\frac{\alpha_k(n)^2}{n}
\ll
r^{-W_0}\sum_{\substack{n\in S(X)}}\frac{d_{|k|}(n)^2r^{\Omega(n)}}{n}
\ll
r^{-W_0}(\log X)^{rk^2}
=
o(1)
\end{equation}
where in the first inequality we have applied Rankin's trick in the form $r^{\Omega(n)-W_0}\geqs 1$. 

Since $\alpha_k(n)=\beta_k(n)$ if $\Omega(n)\leqs W_0$ the main term is
\begin{align*}
\sum_{\substack{n\in S(X)\\\Omega(n)\leqs W_0}}\frac{\alpha_k(n)^2}{n}
 = &
\sum_{\substack{n\in S(X)}}\frac{\beta_k(n)^2}{n}
+
O\bigg( \sum_{\substack{n\in S(X)\\\Omega(n)>W_0}}\frac{|\beta_k(n)|^2}{n}\bigg).
\end{align*}
The bound $|\beta_k(n)|\leqs d_{|k|}(n)$ and the same analysis as in \eqref{tail sum bound} shows that this error term is $o(1)$. From the properties of $\beta_k(n)$ we get   
\begin{align*}
\sum_{\substack{n\in S(X)}}\frac{\beta_k(n)^2}{n}
= &
\prod_{p\leqs X}\bigg(\sum_{m:\, p^m\leqs X}\frac{d_k(p^m)^2}{p^m}+\sum_{m:\, p^m> X}\frac{\beta_k(p^m)^2}{p^m}\bigg)
\\
= &
\prod_{p\leqs X}\sum_{m\geqs 0}\frac{d_k(p^m)^2}{p^m}
\prod_{p\leqs X}\bigg(1+O\Big(\sum_{m:\, p^m> X}\frac{d_{|k|}(p^m)^2}{p^m}\Big)\bigg).
\end{align*}
We split the second product at $\sqrt{X}$ and apply the bound $d_k(n)\ll n^{\epsilon}$ to find that it is 
\begin{equation}\label{prod split}
\prod_{p\leqs \sqrt{X}}\Big(1+O\Big(\frac{1}{X^{1-\epsilon}}\Big)\Big)
\prod_{\sqrt{X}<p\leqs X}\Big(1+O\Big(\frac{1}{p^{2-\epsilon}}\Big)\Big)
=
1+O(X^{-1/2+\epsilon})
\end{equation}
by the prime number theorem. Then by Mertens' theorem we have
\[
\prod_{p\leqs X}\sum_{m\geqs 0}\frac{d_k(p^m)^2}{p^m}
\sim 
a(k)(e^\gamma \log X)^{k^2}
\]
since $a(k)$ is an absolutely convergent product. %This completes the proof of Proposition \ref{P prop}.
\end{proof}

%%%%%%%%%%%% PROP 2 %%%%%%%%%%%%

\section{Asymptotics for the 2nd and 4th moments of the Hadamard product: Proof of Proposition \ref{Z asymp prop}}\label{Hadamard sec}

From \eqref{hybrid formula} we have 
\[
Z_X(\tfrac{1}{2}+it)=\zeta(\tfrac{1}{2}+it)P_X(\tfrac{1}{2}+it)^{-1}\big(1+O(1/\log X)\big)
\]
and thus it suffices to consider the second and fourth moment of the object on the right. Our aim is to first replace $P_X$ by its Dirichlet polynomial approximation and then apply formulas for the twisted second and fourth moments of the zeta function. 

\subsection{The second moment} 
As before, we decompose the integral as 
\begin{equation}\label{int decomp}
\frac1T\int_T^{2T}|Z_X(\tfrac12+it)|^{2}dt
=
\frac1T\int_\mc{S}|Z_X(\tfrac12+it)|^{2}dt
+
\frac1T\int_\mc{E}|Z_X(\tfrac12+it)|^{2}dt.
\end{equation}
Working unconditionally first, we apply the Cauchy--Schwarz inequality to the integral over $\mc{E}$ to find it is 
\[
\ll(\log T)^2 \bigg(\frac1T\int_\mc{E}|P_X(\tfrac12+it)^{-1}|^{4}dt\bigg)^{1/2}
\]
using Ingham's asymptotic for the fourth moment. Since $X\leqs \tfrac{1}{10^4}(\log T)^2(\log_2T)^2$ and $1/10^4<\eta_{2}$ we find that this is $\ll (\log T)^2e^{-\delta V_0}=o(1)$ by Lemma \ref{PE lem}. 

If we assume RH we can apply H\"older's inequality in the form 
\begin{equation}\label{Z holder}
\begin{split}
\frac1T\int_\mc{E}&|Z_X(\tfrac12+it)|^{2}dt
\\
\ll &
 \bigg(\frac1T\int_T^{2T}|\zeta(\tfrac12+it)|^{2(1+{1}/{\epsilon})}dt\bigg)^{\tfrac{\epsilon}{1+\epsilon}}
  \bigg(\frac1T\int_\mc{E}|P_X(\tfrac12+it)|^{-2(1+{\epsilon})}dt\bigg)^{\tfrac{1}{1+\epsilon}}
\end{split}
\end{equation}
for some $\epsilon>0$.
The first term on the right is $\ll (\log T)^{4(1+1/\epsilon)}$ by Harper's \cite{Ha} conditional bound $\int_T^{2T}|\zeta(1/2+it)|^{2k}\ll T(\log T)^{k^2}$. Since $X\leqs (\log T)^{\sqrt{6}-\epsilon^\prime}$ and $\sqrt{6}-\epsilon^\prime\leqs \theta_{1+\epsilon}-\epsilon^\prime/2$ on choosing $\epsilon$ small enough, the second term on the right is $\ll e^{-\delta V_0}$ by Lemma \ref{PE lem}. Therefore, the quantity in \eqref{Z holder} is $o(1)$.

By Lemma \ref{P lem} we have 
\begin{align*}
\frac1T\int_\mc{S}|Z_X(\tfrac12+it)|^{2}dt
\sim &
\frac1T\int_\mc{S}|\zeta(\tfrac12+it)|^{2}|D(t,-1)|^2 dt
\end{align*}
which we write as
\[
\frac1T\int_T^{2T}|\zeta(\tfrac12+it)|^{2}|D(t,-1)|^2 dt
+
O\bigg(\frac1T\int_\mc{E}|\zeta(\tfrac12+it)|^{2}|D(t,-1)|^2 dt\bigg).
\]
Applying the Cauchy--Schwarz inequality twice along with Ingham's fourth moment bound, Lemma \ref{E meas lem} and \eqref{D mth moment} we find that the error term here is $o(1)$. 

It remains to show that 
\[
I_1
:=
\frac1T\int_T^{2T}|\zeta(\tfrac12+it)|^{2}|D(t,-1)|^2 dt
\sim
\frac{\log T}{e^\gamma \log X}.
\]
The mean square of the zeta function times an arbitrary Dirichlet polynomial has been computed before e.g. see \cite{BCH,BCR}. From there we see that 
\[
I_1=\sum_{m,n\in S(X)}\frac{\alpha_{-1}(m)\alpha_{-1}(n)(m,n)}{mn} \log \Big( \frac{B T (m,n)^2}{mn} \Big)+o(1)
\] 
for some constant $B$.
%On using 
%$$ 
%\log \Big(  \frac{B T (m,n)^2}{mn} \Big) = \frac{1}{2\pi i } \int_{|z|= \varrho} \Big( \frac{BT (m,n)^2}{mn}\Big)^z \frac{dz}{z^2}, 
%$$
%with $\varrho=1/\log T$ we may write
%$$ 
%I_1=\frac{1}{2\pi i } \int_{|z|= \varrho}\sum_{m, n\in S(X)} \frac{\alpha_{-1}(m)\alpha_{-1}(n)(m,n)^{1+2z}}{(mn)^{1+z}} (BT)^z  \frac{dz}{z^2}+o(1). 
%$$ 
On applying the bound $\alpha_{-1}(n)\ll d_{1}(n)\ll1$ and following the argument in the proof of Theorem 3, pg. 530, of \cite{GHK} we easily see that 
 \[
 \sum_{m,n\in S(X)}\frac{\alpha_{-1}(m)\alpha_{-1}(n)(m,n)}{mn} \log \Big( \frac{B (m,n)^2}{mn} \Big)\ll (\log X)^{10}
 \]
 and thus we can consider the remaining sum.
 
As in the previous section, we first  estimate the sum over integers for which $\Omega(m),\Omega(n)>W_0$ (which equals $10\log_2 T\log_3 T$ in this case). Applying the bound $\alpha_{-1}(n)\ll1$ leads to an error of size 
\[
\ll
\log T
\sum_{\substack{m, n\in S(X)\\\Omega(m)>W_0}} \frac{(m,n)}{mn}
\ll 
(\log T) r^{-W_0} 
\sum_{\substack{m, n\in S(X)}} \frac{(m,n)r^{\Omega(m)}}{mn}
\]
for any $1<r<2$. A short computation shows this last sum is 
\[
\prod_{p\leqs X}\Big(1+{(2r+1)}{p^{-1}}+O(p^{-2})\Big)\ll (\log X)^{2r+1}
\]
 and so the terms with $\Omega(m),\Omega(n)>W_0$ contribute an error of size $o(1)$. 
 
 In the main term we replace $\alpha_{-1}(n)$ with $\beta_{-1}(n)$ and then re-extend the sum to include those integers for which $\Omega(m),\Omega(n)>W_0$. By the bounds  $\beta_{-1}(n)\ll d_{1}(n)\ll 1$, the same argument shows that this introduces an error of $o(1)$. Thus,
\[
I_1=\log T\sum_{m,n\in S(X)}\frac{\beta_{-1}(m)\beta_{-1}(n)(m,n)}{mn} +O((\log X)^{10}).
\] 
By symmetry and the properties of $\beta_k(n)$  we find that the sum is 
\begin{align*}
&
\prod_{p\leqs X}
\bigg(
\sum_{\substack{m,n\geqs 0\\p^m,\,\,p^n\leqs X}}\frac{\mu(p^m)\mu(p^n)}{p^{m+n-\min(m,n)}}
+
O\Big(\sum_{\substack{m,n\geqs 0\\p^m> X}}\frac{1}{p^{m+n-\min(m,n)}}\Big)
\bigg)
\\
= &
\prod_{p\leqs X}\Big(1-\frac{1}{p}\Big)
\prod_{p\leqs X}
\bigg(
1
+
O\Big(\sum_{\substack{m,n\geqs 0\\p^m> X}}\frac{1}{p^{m+n-\min(m,n)}}\Big).
\bigg)
\end{align*}
To estimate the second product we note that the sum in the error term is $\ll \sum_{p^m>X}(m+1)p^{-m}\ll (\log X)p^{-\lceil \log X/\log p\rceil}$
and then  split the product at $\sqrt{X}$, as before. In this way we find it is $1+O(X^{-1/2+\epsilon})$ and therefore by Mertens' Theorem 
\[
I_1\sim\frac{\log T}{e^\gamma \log X}
\] 
as desired.
%%%%%%%% FOURTH MOMENT %%%%%%%%

\subsection{The fourth moment: Initial clearing} Not surprisingly,  the fourth moment requires more work in both the initial stages and the arithmetic computations. Our aim is to show that 
\[
I_2:=\frac{1}{T}\int_T^{2T}|\zeta(\tfrac12+it)|^4|P_X(\tfrac{1}{2}+it)|^{-4}dt
\sim
\frac{1}{12}\bigg(\frac{\log T}{e^\gamma\log X}\bigg)^4.
\]
In this subsection the goal is to replace $P_X(1/2+it)^{-2}$ by $D(t,-2)$. 

Splitting the integral as in \eqref{int decomp} we see that our first task is to bound 
\begin{equation}\label{integral}
\frac{1}{T}\int_\mc{E}|\zeta(\tfrac12+it)|^4|P_X(\tfrac{1}{2}+it)|^{-4}dt.
\end{equation}
On RH we can deal with this by applying H\"older's inequality as in \eqref{Z holder}. Following the same argument and using the fact that $\sqrt{5}-\epsilon^\prime\leqs \theta_{2(1+\epsilon)}-\epsilon^\prime/2$ shows that this is $o(1)$.

To bound this unconditionally requires more work. First note that since
\[
\mc{E}\subset \mc{E}^\prime:=\Big\{t\in[T,2T]: \Big|\sum_{n\leqs X}\frac{\Lambda(n)}{n^{1/2+it}\log n}\Big|\geqs V_0\Big\}
\] 
we can upper bound by the integral over $\mc{E}^\prime$. Let $V_j=e^jV_0$  and define $\mc{J}$ to be the maximal $j$ such that $V_\mc{J}\leqs V_{\max}$. Let 
\[
\mc{E}_j
=
\Big\{t\in [T,2T]: V_j\leqs \Big|\sum_{n\leqs X}\frac{\Lambda(n)}{n^{1/2+it}\log n}\Big|\leqs V_{j+1}\Big\}
\]
so that 
\[
\mc{E}^\prime
=
\cup_{j=0}^\mc{J}\mc{E}_j.
\]
Then 
\begin{equation}\label{fourth ER}
\begin{split}
&\frac{1}{T}\int_{\mc{E}^\prime}|\zeta(\tfrac12+it)|^4|P_X(\tfrac{1}{2}+it)|^{-4}dt
\leqs 
\sum_{j=0}^\mc{J} e^{4V_{j+1}}\frac{1}{T}\int_{\mc{E}_j}|\zeta(\tfrac12+it)|^4dt
\\
\leqs &
\sum_{j=0}^\mc{J} e^{4V_{j+1}}V_j^{-2r_j}\frac{1}{T}\int_{T}^{2T}|\zeta(\tfrac12+it)|^4\Big|\sum_{n\leqs X}\frac{\Lambda(n)}{n^{1/2+it}\log n}\Big|^{2r_j}dt
\end{split}
\end{equation}
for any given integer  $r_j\geqs 0$. The combinatorics are simplified if we focus on the prime sums so we apply Jensen's inequality in the form 
\[
\Big|\sum_{n\leqs X}\frac{\Lambda(n)}{n^{1/2+it}\log n}\Big|^{2r_j}
\leqs 
9^{r_j}\bigg(
\Big|\sum_{p\leqs X}\frac{1}{p^{1/2+it}}\Big|^{2r_j}
+
\Big|\sum_{p\leqs \sqrt{X}}\frac{1}{2p^{1+2it}}\Big|^{2r_j}
+
O(C^{2r_j})\bigg).
\]
It will be clear after the computations that the first sum here gives the dominant contribution and so we focus on this. Note that by the multinomial theorem,
\[
\Big(\sum_{p\leqs X}\frac{1}{p^{1/2+it}}\Big)^{r_j}
=
r_j!\sum_{\substack{n\leqs X^{r_j}\\\Omega(n)=r_j}}\frac{\mf{g}(n)}{n^{1/2+it}}
\]
where $\mf{g}(n)$ is the multiplicative function satisfying $\mf{g}(p^\alpha)=1/\alpha!$. Accordingly,  \eqref{fourth ER} is 
\begin{equation}\label{integral bound}
\ll 
\sum_{j=0}^\mc{J} e^{4V_{j+1}}9^{r_j}V_j^{-2r_j}\cdot
\frac{1}{T}\int_{T}^{2T}|\zeta(\tfrac12+it)|^4
\Big|r_j!\sum_{\substack{n\leqs X^{r_j}\\\Omega(n)=r_j}}\frac{\mf{g}(n)}{n^{1/2+it}}\Big|^{2}dt
\end{equation}

The twisted fourth moment of the zeta function has been computed before \cite{BBLR,HY} and has been applied in similar situations \cite{HRS}. 
Provided $X^{r_j}\leqs T^{1/4-\epsilon}$ i.e. 
\begin{equation}\label{r condition}
r_j\leqs (1/4-\epsilon)\frac{\log T}{\log X},
\end{equation}
Proposition 4 and formula (8) of \cite{HRS} (see section 6 there) give
\begin{equation*}
\frac{1}{T}\int_{T}^{2T}|\zeta(\tfrac12+it)|^4
\Big|r_j!\sum_{\substack{n\leqs X^{r_j}\\\Omega(n)=r_j}}\frac{\mf{g}(n)}{n^{1/2+it}}\Big|^{2}dt
\ll 
(\log T)^4
9^{r_j}
\sum_{\substack{m,n\leqs X^{r_j}\\\Omega(m)=\Omega(n)=r_j}} \frac{r_j!^2\mf{g}(n) \mf{g}(m)}{[n,m]}. 
\end{equation*}
Following the arguments of \cite{HRS} which lead to formula (9) there, we find that this is 
\begin{equation}\label{sne bound}
\ll 
(\log T)^4 9^{r_j}r_j!\Big(\sum_{p\leqs X}\frac{1}{p}\Big)^{r_j}\exp\Big(\sum_{p\leqs X}\frac{1}{p}\Big)
\ll
(\log T)^5 r_j^{1/2}\Big(\frac{9r_j\log\log X}{e}\Big)^{r_j}.
\end{equation}
We choose 
\[
r_j=\frac{12V_j}{\log V_j}\leqs\frac{12V_{\max}}{\log V_{\max}}\lesssim \frac{48X^{1/2}}{(\log X)^2}\lesssim \frac{24\log T}{100\log X} \qquad\forall j
\]
since $X\leqs \tfrac{1}{10^4}(\log T)^2(\log_2 T)^2$. Clearly, this satisfies \eqref{r condition}.  Then applying \eqref{sne bound} in \eqref{integral bound} we find that 
\begin{align*}
\frac{1}{T}\int_{\mc{E}^\prime}&|\zeta(\tfrac12+it)|^4|P_X(\tfrac{1}{2}+it)|^{-4}dt
\\
 \ll &
  (\log T)^5\sum_{j=0}^\mc{J} e^{4V_{j+1}}\Big(\frac{C\log\log X}{V_j\log V_j}\Big)^{12V_j/\log V_j}
\\ 
\ll &
 (\log T)^5\sum_{j=0}^\mc{J} e^{4V_{j+1}-(12-o(1))V_j}
 \ll
  (\log T)^5\sum_{j=0}^\mc{J} e^{-e^jV_{0}(12-4e+o(1))}=o(1).
\end{align*}

We have therefore arrived at
\[
I_2
=
(1+o(1))
\frac{1}{T}\int_\mc{S}|\zeta(\tfrac12+it)|^4|D(t,-2)|^2dt+o(1)
\]
on applying Lemma \ref{P lem} in the integral over $\mc{S}$. After extending to the full range of integration $[T,2T]$ it remains to estimate 
\[
\frac{1}{T}\int_\mc{E}|\zeta(\tfrac12+it)|^4|D(t,-2)|^2dt. 
\]
However, from the definition of $D(t,k)$ we have 
\[
|D(t,k)|
\leqs 
\sum_{j=0}^{W_0} \frac{|k|^j}{j!}\Big|\sum_{n\leqs X}\frac{\Lambda(n)}{n^{1/2+it}\log n}\Big|^j 
\leqs 
\exp\Big(\Big|k\sum_{n\leqs X}\frac{\Lambda(n)}{n^{1/2+it}\log n}\Big|\Big). 
\]
 and so we can apply the same argument as above to acquire the bound $o(1)$ for this integral. Thus we have 

\[I_2\sim J_2\]
where 
\[
J_2
:=\frac{1}{T}\int_T^{2T}|\zeta(\tfrac12+it)|^4|D(t,-2)|^2dt.
\]
%%%%%%%%%%%% FOURTH MOMENT ARITHMETIC %%%%%%%%%%

\subsection{The fourth moment: Arithmetic computations}

In this section our aim is to show that 
\[J_2
\sim
\frac{1}{12}\bigg(\frac{\log T}{e^\gamma\log X}\bigg)^4.
\]
The formulas for the twisted fourth moment of the zeta function given in the literature \cite{BBLR,HY} apply to smoothed integrals and accordingly we must first smooth $J_2$. Let $\Phi_-,\Phi_+$ be smooth approximations of compact support satisfying 
\begin{equation}\label{encadrement}
\Phi_-(t)\leqs \mathds{1}_{t\in [T,2T]}\leqs \Phi_+(t)
\end{equation}
with derivatives $\Phi_{\pm}^{(j)}(t)\ll T^\epsilon$. For example, we may take $\Phi_-$ to be compactly supported on $[1,2]$ and equal to $1$ on the interval $[1+T^{-\epsilon},2-T^{-\epsilon}]$ with smooth, monotonic decay to zero at each endpoint. Then, on letting $\Phi$ be either $\Phi_-$ or $\Phi_+$ we consider the smoothed integral 
\[
J_{2,\Phi}:=\frac{1}{T}\int_\mathbb{R}\Phi\Big(\frac{t}{T}\Big)|\zeta(\tfrac12+it)|^4|D(t,-2)|^2dt.
\]
We note that the error incurred from these approximations will be $\ll T^{1-\epsilon}$ which is tolerable given the asymptotic we seek.

\begin{thm}[Theorem 1.2 of \cite{BBLR}]\label{twist thm}Let $\Phi$ be as above and $\alpha_1,\alpha_2,\alpha_3,\alpha_4\ll 1/\log T$.  Let $\Xi$ be the subgroup of $S_4$ consisting  of the identity, those permutations which swap just one element of $\{1,2\}$ with $\{3,4\}$ and the permutation satisfying $\tau(1)=3,\tau(2)=4$. Then for any Dirichlet polynomial $\sum_{n\leqs y}a(n)n^{-s}$ satisfying $y\leqs T^{1/4-\epsilon}$ and $a(n)\ll n^\epsilon$
we have 
\begin{multline*}
\int_\mathbb{R}\zeta(\tfrac12+\alpha_1+it)\zeta(\tfrac12+\alpha_2+it)\zeta(\tfrac12-\alpha_3-it)\zeta(\tfrac12-\alpha_4-it)\Big|\sum_{n\leqs y}\frac{a(m)}{m^{1/2 +it}}\Big|^2\Phi\Big(\frac{t}{T}\Big)dt
\\
=\sum_{m_1,m_2\leqs y}a(m_1)\overline{a(m_2)}\int_{\mathbb{R}}\Phi\Big(\frac{t}{T}\Big)\sum_{\tau\in\Xi}\Big(\frac{t}{2\pi}\Big)^{\sum_{j=1}^4\alpha_{\tau(j)}-\alpha_j}
Z_{\tau(\alpha_1),\tau(\alpha_2),\tau(\alpha_3),\tau(\alpha_4),m_1,m_2}
dt
\\
+O(T^{1-\epsilon})
\end{multline*}
where 
\[
Z_{\alpha_1,\alpha_2,\alpha_3,\alpha_4,m_1,m_2}
=
\sum_{m_1n_1n_2=m_2n_3n_4}\frac{1}{(m_1m_2)^{1/2}n_1^{1/2+\alpha_1}n_2^{1/2+\alpha_2}n_3^{1/2-\alpha_3}n_4^{1/2-\alpha_4}}V\Big(\frac{n_1n_2n_3n_4}{t^2}\Big)
\]
and 
\[
\qquad V(x)=\frac{1}{2\pi i }\int_{c-i\infty}^{c+i\infty}\frac{G(s)}{s}(4\pi^2x)^{-s}ds,\qquad c>0
\]
with $G(s)$ an even function of rapid decay in vertical strips satisfying $G(0)=1$.  
\end{thm}

\begin{rem}
We remark that the choice of function $G(s)$ is flexible and it can be prescribed to have zeros at linear combinations of the shifts. This is fairly typical and is used to cancel unnecessary poles later on. We will take $G(s)=Q_{{\alpha}}(s)\exp(s^2)$ where $Q_{{\alpha}}(s)$ is an even polynomial which is 1 at $s=0$ and zero at $2s=\alpha_3-\alpha_1,\alpha_4-\alpha_1,\alpha_3-\alpha_2$ and $\alpha_4-\alpha_2$. Note these conditions on $Q$ imply that for fixed $\Re(s)$,
\begin{equation}\label{G bound}
G(s)\ll (\log T)^4 e^{-\Im(s)^2}
\end{equation}
since $\alpha_j\ll 1/\log T$.
\end{rem}
 Let us compute term corresponding to the identity: $\tau=\text{id}$. Denote this by
\begin{align*}
\mc{K}
=
\mc{K}_{\alpha}(t,X)
:= &
\sum_{m_1,m_2\in S(X)}\alpha_{-2}(m_1)\alpha_{-2}(m_2)Z_{\alpha_1,\alpha_2,\alpha_3,\alpha_4,m_1,m_2}
\\
= &
\sum_{\substack{m_1n_1n_2=m_2n_3n_4\\m_j\in S(X)}}\frac{\alpha_{-2}(m_1)\alpha_{-2}(m_2)}{(m_1m_2)^{1/2}n_1^{1/2+\alpha_1}n_2^{1/2+\alpha_2}n_3^{1/2-\alpha_3}n_4^{1/2-\alpha_4}}
V\Big(\frac{n_1n_2n_3n_4}{t^2}\Big).
\end{align*}
By shifting the contour in $V$ to the either the left or right depending on whether $x\ll 1$ or $x\gg 1$, respectively, we find that $V(x)\ll (1+|x|)^{-A}$ for any $A>0$. Accordingly, on applying the bounds $\alpha_k(m)\ll m^{\epsilon}\cdot \mathds{1}_{m\leqs T^{\delta}}$ we may restrict the above sum to those $n_j$ satisfying $n_1n_2n_3n_4\ll t^{2+\epsilon}\ll T^{2+\epsilon}$ at the cost of an error of size $o(1)$. Then the contribution from those $m_1$ with $\Omega(m_1)>W_0$ is for $1<r<2$
\[
\ll r^{-W_0}\sum_{m\leqs T^{2+\epsilon}}\frac{r^{\Omega(n)}d_4(m)^2}{m}
\ll
r^{-W_0}(\log T)^{16r}
=
o(1) 
\]
where for the first inequality we have applied Rankin's trick in the form $r^{\Omega(m)-W_0}\geqs 1$ along with the bound $\alpha_{-2}(n)\ll d(n)$. The same bound holds for the sum over $\Omega(m_2)>W_0$. Then on replacing $\alpha_{-2}(m)$ with $\beta_{-2}(m)$ and re-extending the sums (which by the same arguments incurs an error of $o(1)$) we have 
\[
\mc{K}
=
\sum_{\substack{m_1n_1n_2=m_2n_3n_4\\m_j\in S(X)}}\frac{\beta_{-2}(m_1)\beta_{-2}(m_2)}{(m_1m_2)^{1/2}n_1^{1/2+\alpha_1}n_2^{1/2+\alpha_2}n_3^{1/2-\alpha_3}n_4^{1/2-\alpha_4}}V\Big(\frac{n_1n_2n_3n_4}{t^2}\Big)+o(1).
\]

Unfolding the integral for $V(x)$ and pushing the sum through we find 
\begin{equation}\label{K}
\mc{K}
=
\frac{1}{2\pi i}\int_{c-i\infty}^{c+i\infty}\mc{F}_{\alpha,X}(s)
\frac{G(s)}{s}\bigg(\frac{t}{2\pi}\bigg)^{2s}ds
+
o(1)
\end{equation}
where 
\begin{align*}
\mc{F}_{\alpha,X}(s)
= &
\sum_{\substack{m_1n_1n_2=m_2n_3n_4\\m_j\in S(X)}}\frac{\beta_{-2}(m_1)\beta_{-2}(m_2)}{(m_1m_2)^{1/2}n_1^{1/2+\alpha_1+s}n_2^{1/2+\alpha_2+s}n_3^{1/2-\alpha_3+s}n_4^{1/2-\alpha_4+s}}
\\
= &
\sum_{\substack{m_1n_1=m_2n_2\\m_j\in S(X)}}\frac{\beta_{-2}(m_1)\beta_{-2}(m_2)\sigma_{\alpha_1,\alpha_2}(n_1)\sigma_{-\alpha_3,-\alpha_4}(n_2)}{(m_1m_2)^{1/2}(n_1n_2)^{1/2+s}}
\end{align*}
with
 $\sigma_{u,v}(n)=\sum_{d_1d_2=n}d_1^{-u}d_2^{-v}$. Expressing this as an Euler product %and using the identity
%\begin{multline*}
 %\sum_{n\geqs 0}\frac{\sigma_{\alpha_1,\alpha_2}(p^{n})\sigma_{-\alpha_3,-\alpha_4}(p^{n})}{p^{n(1+2s)}}
%\\
% =
 %\frac{1-p^{-2-\alpha_1-\alpha_2+\alpha_3+\alpha_4-4s}}{(1-p^{-1-\alpha_1+\alpha_4-2s})(1-p^{-1-\alpha_2+\alpha_3-2s})(1-p^{-1-\alpha_2+\alpha_3-2s})(1-p^{-1-\alpha_2+\alpha_4-2s})}
%\end{multline*}
 we have
\begin{align*}
\mc{F}_{\alpha,X}(s)
= &
\mc{A}_{\alpha}(s)
\mc{G}_{\alpha,X}(s)
\end{align*}
where 
\begin{multline*}
\mc{A}_{\alpha}(s)
=
\\
\frac{\zeta(1+\alpha_1-\alpha_3+2s)\zeta(1+\alpha_1-\alpha_4+2s)\zeta(1+\alpha_2-\alpha_3+2s)\zeta(1+\alpha_2-\alpha_4+2s)}
{\zeta(2+\alpha_1+\alpha_2-\alpha_3-\alpha_4+4s)}
\end{multline*}
and
\begin{multline*}
\mc{G}_{\alpha,X}(s)
=
\prod_{p\leqs X}
\sum_{\substack{m_1+n_1\\=m_2+n_2}}\frac{\beta_{-2}(p^{m_1})\beta_{-2}(p^{m_2})\sigma_{\alpha_1,\alpha_2}(p^{n_1})\sigma_{-\alpha_3,-\alpha_4}(p^{n_2})}{p^{\tfrac{1}{2}(m_1+m_2)+(\tfrac{1}{2}+s)(n_1+n_2)}}
\\
\slash \sum_{n\geqs 0}\frac{\sigma_{\alpha_1,\alpha_2}(p^{n})\sigma_{-\alpha_3,-\alpha_4}(p^{n})}{p^{n(1+2s)}}.
\end{multline*}

Shifting the line of integration in \eqref{K} to $\Re(s)=-1/\log X$ we pick up a simple pole only at $s=0$ (the poles of $A_\alpha(s)$ being cancelled by the zeros of $G(s)$). Since $\beta_{-2}(n)\ll d(n)$ and $\sigma_{\alpha_i,\alpha_j}(p^n)\ll p^{n/\log T}d(p^n)$ we find that on the new line of integration 
\[
\mc{G}_{\alpha,X}(s)\ll (\log X)^{O(1)}.
\]
Therefore, on combining this with the bound for $G(s)$ given in \eqref{G bound} we see the integral over the new line is bounded by
\[
\ll t^{-2/\log X}(\log T)^{O(1)}=o(1)
\]
since $t\asymp T$. Hence
\[
\mc{K}_\alpha = \mc{A}_{\alpha}(0)\mc{G}_{\alpha,X}(0)+o(1).
\]

We have satisfactorily computed the contribution from a single $Z$ term and thus it remains to find the combinatorial sum of these which appears in Theorem \ref{twist thm}. Using the results of \cite{cfkrs} we can express this sum as a multiple contour integral. Precisely, Lemma 2.5.1 there gives 
  \begin{multline*}
\sum_{\tau\in\Xi}\Big(\frac{t}{2\pi}\Big)^{\sum_{j=1}^4\alpha_{\tau(j)}-\alpha_j}
\mc{K}_{\tau(\alpha)}
 \\
 = 
 \frac{1}{4(2\pi i)^4} \int_{\substack{ |z_j| = 3^j / \log T \\ 1\leqs j\leqs 4}} \frac{\mc{A}_{z_1, z_2, z_3, z_4}(0)\mc{G}_{z_1, z_2, z_3, z_4,X}(0) \Delta(z_1, z_2, z_3, z_4)^2}{ \prod_{i,j = 1}^{4} (z_i-\alpha_j)}
 \\
\times \Big ( \frac{t}{2\pi} \Big )^{\sum_{j=1}^2(z_{j+2}-z_j)/2} d\underline{z}+o(1)
  \end{multline*}
where $\Delta(\underline{z})$ denotes the vandermonde determinant.  A short calculation shows that 
\[
\frac{\partial}{\partial z_j}\mc{G}_{\underline{z},X}(0)\bigg|_{\underline{z}=\underline{0}}\ll \mc{G}_{\ul{0},X}(0)\sum_{p\leqs X}\frac{\log p}{p}
\ll
\mc{G}_{\ul{0},X}(0)\log X
\]
and hence we acquire the Taylor expansion 
\[
\mc{G}_{z_1, z_2, z_3, z_4,X}(0) 
=
\mc{G}_{\underline{0},X}(0)\Big( 1 
+
O\Big(\log X\sum_{j=1}^4 |z_j|\Big) 
\Big)
\]
whilst from the Laurent expansion of the zeta function we get
\[
\mc{A}_{z_1, z_2, z_3, z_4,X}(0) 
=\frac{1}{\zeta(2)}\prod_{i,j=1}^2\frac{1}{(z_i-z_{j+2})}\Big( 1 
+
O\Big(\sum_{j=1}^4 |z_j|\Big) \Big).
\]
On setting the shifts $\alpha_j$ equal to zero, substituting $z_j\mapsto z_j/\tfrac12\log(t/2\pi )$, and applying these expansions we find 
\[
J_{2,\Phi}
=
\int_\mathbb{R}\Phi\Big(\frac{t}{T}\Big)\Big(c_4\log^4(t/2\pi)\cdot \frac{\mc{G}_{\underline{0},X}(0)}{\zeta(2)}\Big(1+O\Big(\frac{\log X}{\log T}\Big)\Big)dt +o(T)
\]
where 
\[
c_4
=
 \frac{1}{4\cdot 2^4(2\pi i)^4} \int_{\substack{ |z_j| = 3^j \\ 1\leqs j\leqs 4}} \frac{\Delta(z_1, z_2, z_3, z_4)^2}{ \prod_{i,j = 1}^{2} (z_i-z_{j+2})}
e^{\sum_{j=1}^2z_{j+2}-z_j} \prod_{j=1}^4\frac{d{z_j}}{z_j^4}.
\]

From section 2.7 of \cite{cfkrs} we know that $c_4=g(2)=1/12$ where $g(k)$ is given by \eqref{g(k)}.
Furthermore, 
\begin{align*}
\mc{G}_{\underline{0},X}(0)
=
\prod_{p\leqs X}
\sum_{\substack{m_1+n_1\\=m_2+n_2}}\frac{\beta_{-2}(p^{m_1})\beta_{-2}(p^{m_2})d(p^{n_1})d(p^{n_2})}{p^{\tfrac{1}{2}(m_1+m_2+n_1+n_2)}}
\slash \sum_{n\geqs 0}\frac{d(p^{n})^2}{p^{n}}.
\end{align*}
The denominator here is 
\[
\prod_{p\leqs X} \bigg(\sum_{n\geqs 0}\frac{d(p^{n})^2}{p^{n}}\bigg)^{-1}
=
\prod_{p\leqs X}\frac{(1-p^{-1})^4}{1-p^{-2}}
\sim
\frac{\zeta(2)}{(e^\gamma\log X)^4}
\]
whereas the numerator is 
\begin{align*}
\mc{G}_{\underline{0},X}(0)
%= &
%\prod_{p\leqs X}
%\sum_{\substack{m_1+n_1\\=m_2+n_2}}\frac{d_{-2}(p^{m_1})d_{-2}(p^{m_2})d(p^{n_1})d(p^{n_2})}{p^{\tfrac{1}{2}(m_1+m_2+n_1+n_2)}}
%\\
= &
\prod_{p\leqs X}
\bigg(1+O\Big(\sum_{\substack{m_1+n_1\\=m_2+n_2\\ p^{m_1}> X}}\frac{d(p^{m_1})d(p^{m_2})d(p^{n_1})d(p^{n_2})}{p^{\tfrac{1}{2}(m_1+m_2+n_1+n_2)}}\Big)\bigg)
\end{align*}
since $\beta_{-2}(p^m)=d_{-2}(p^m)$ for $p^m\leqs X$ and $\beta_{-2}(p^m)\ll d(p^m)$ in general. Then since $n_1\geqs 0$, the sum in the error is $\ll \sum_{m:\, p^m>X} d_4(p^m)^2/p^m$ after forming the convolution. Therefore, we can apply the same argument which gave \eqref{prod split} to find that the numerator is $1+O(X^{-1/2+\epsilon})$. Consequently, we have 
\[
J_{2,\Phi}
\sim
 \frac{\hat{\Phi}(0)}{12}\cdot T\bigg(\frac{\log T}{e^\gamma\log X}\bigg)^4
=
 \frac{T}{12}\bigg(\frac{\log T}{e^\gamma\log X}\bigg)^4+O(T^{1-\epsilon})
\]
and so the result follows by \eqref{encadrement}.

%%%%%%%%%%% PROP 3 INiTIAL %%%%%%%%%%%%

\section{Some useful tools}\label{tools sec} 

In this short section we describe some tools which will come in handy when proving Proposition \ref{Z prop}. The first relates to the exponential truncation of a more general prime sum and will be used extensively throughout.

Given a general Dirichlet polynomial of the form $D(s)=\sum_{p\leqs Y}a(p)p^{-s}$, suppose $t\in [T,2T]$ is such that $|kD(s)|\leqs Z$. Then by \eqref{exp trunc 0} we have 
 \begin{equation}
 \label{exp trunc}
 \bigg|e^{kD(s)}-\sum_{0\leqs j\leqs 10Z}\frac{(kD(s))^j}{j!}\bigg|
 \leqs 
 e^{-10Z} . 
 \end{equation}
By the multinomial theorem the truncated exponential series can be written as 
\begin{equation}\label{trunc form}
\sum_{0\leqs j\leqs 10Z}\frac{1}{j!}\bigg(k\sum_{ p\leqs Y}\frac{a(p)}{p^{s}}\bigg)^j
=
\sum_{\substack{\Omega(n)\leqs 10Z\\ p|n\implies  p\leqs Y}}\frac{k^{\Omega(n)}a(n)\mf{g}(n)}{n^s}
\end{equation}
where we recall $\mf{g}$ is the multiplicative function such that $\mf{g}(p^\alpha)=1/\alpha!$, and $a(n)$ is the completely multiplicative extension of $a(p)$. Observe this is a Dirichlet polynomial of length $Y^{10Z}$.

 Our remaining observations relate to mean values of Dirichlet polynomials. We first state the mean value theorem of Montgomery and Vaughan \cite{MV} which gives  for any complex coefficients $a(n)$,
\begin{equation}\label{MVT}
\frac{1}{T}\int_T^{2T}\Big|\sum_{n\leqs N}\frac{a(n)}{n^{it}}\Big|^2dt
=
(1+O(N/T))\sum_{n\leqs N}|a(n)|^2.
\end{equation}

Suppose we are given $R$ Dirichlet polynomials 
\[
A_j(s) = \sum_{n\in {\mathcal S}_j} a_j(n) n^{-s}, 
\]
where the $\prod_{j=1}^{R} n_j \leqs N=o(T)$ for all $n_j \in {\mathcal S}_j$ i.e. the product of the $A_j(s)$ is short. 
Then the Montgomery--Vaughan mean value theorem readily implies
\begin{align}\label{prod mv}
\frac{1}{T}\int_T^{2T}\prod_{j=1}^R \big|A_j(it)\big|^2dt
\sim&
\sum_{n\leqs N}\Big|\sum_{\substack{n=n_1\cdots n_R\\n_j\in S_j}}a_1(n_1)\cdots a_R(n_R)\Big|^2 \nonumber
\\
= &
\sum_{\substack{n_1\cdots n_R=n_{R+1}\cdots n_{2R}\\n_j\in S_j}}a_1(n_1)\cdots a_R(n_R)\overline{ a_1(n_{R+1})\cdots a_R(n_{2R})}.
\end{align}
Suppose in addition that for any $j_1,j_2$ with $j_1\neq j_2$ the elements of $S_{j_1}$ are all coprime to the elements of $S_{j_2}$. Then
there is at most one way to write $n= \prod_{j=1}^{R} n_j$ with $n_j \in {\mathcal S}_j$ and thus several applications of the mean value theorem imply
\begin{align} 
\label{4.6}
\frac{1}{T} \int_T^{2T} \prod_{j=1}^{R} |A_j(it)|^2 dt 
&= (1+O(NT^{-1})) \sum_{n\le N} \Big| \sum_{\substack{ n= n_1 \cdots n_R \\ n_j\in {\mathcal S}_j} }\prod_{j=1}^{R} a_j(n_j) \Big|^2  \nonumber
\\
&= (1+O(NT^{-1})) \prod_{j=1}^R \Big( \sum_{n_j \in {\mathcal S}_j} |a_j(n_j)|^2 \Big) \nonumber \\ 
&= (1+ O(NT^{-1} ))^{1-R} \prod_{j=1}^{R} \Big( \frac 1T \int_{T}^{2T} |A_j(it)|^2 dt \Big). 
\end{align}
We now move on to proving the upper bound of Proposition \ref{Z prop}.

%%%%%%%%%%%%%%%%%% UPPER BOUND %%%%%%%%%%%%%

\section{The upper bound of Proposition \ref{Z prop}}\label{upper sec}

\subsection{Initial cleaning} 
In this section we are required to show that on RH,
\[
\frac1T\int_T^{2T}|Z_X(\tfrac{1}{2}+it)|^{2k}dt
%\sim
%\frac1T\int_T^{2T}|\zeta(\tfrac12 +it)|^{2k}|P_X(\tfrac{1}{2}+it)|^{-2k}dt
\ll
\bigg(\frac{\log T}{\log X}\bigg)^{k^2}.
\]
On assuming RH, it is a simple task to replace $|P_X(1/2+it)|^{-2k}$ by $|D(t,-k)|^2$ on the left hand side. Indeed; from Harper's \cite{Ha} conditional bound $\int_T^{2T}|\zeta(1/2+it)|^{2k}\ll T(\log T)^{k^2}$, the bound for the moments of $D(t,k)$ in \eqref{D mth moment}, Lemmas \ref{E meas lem}, \ref{PE lem}, \ref{P lem}, and the usual arguments involving the decomposition $[T,2T]=\mc{S}\cup\mc{E}$ along with H\"older's inequality we have 
\begin{equation}
\label{zeta D int}
\frac1T\int_T^{2T}|Z_X(\tfrac{1}{2}+it)|^{2k}dt
\sim
\frac1T\int_T^{2T}|\zeta(\tfrac{1}{2}+it)|^{2k}|D(t,-k)|^2dt
\end{equation}
for all $k\geqs 0$ on RH.
 
 To bound the right hand side we use an upper bound for $\zeta(1/2+it)$ which incorporates the recent developments of Harper \cite{Ha} on moments of the zeta function, although we present the result more in the style of Radziwi\l\l--Soundararajan \cite{RS1} (see the key inequality of section 3 there). Such a treatment is similar to that of \cite{LR}. 

\subsection{An upper bound for $\zeta(\tfrac12+it)$}  We start with a proposition of Soundararajan in a mildly adapted form of Harper.

\begin{lem}
\label{sound lem}
Assume RH. Let $t\in[T,2T]$ be large and suppose $4\leqs x\leqs T^2$. Then 
\[
\log|\zeta(\tfrac{1}{2}+it)|
\leqs 
\Re\sum_{p\leqs x}\frac{1}{p^{{1}/{2}+1/\log x+it}}\frac{\log(x/p)}{\log x}
+\Re\sum_{p\leqs \min(\sqrt{x},\log T)}\frac{1}{2p^{1+2it}}
+\frac{\log T}{\log x}+O(1).
\]
\end{lem}  
\begin{proof}This is Proposition 1 of \cite{Ha}. 
\end{proof}

For the splitting of the prime sums we denote 
\[
T_{-1}=A,\,\,\,\,\,\, T_j=T^{\frac{e^{j}}{(\log\log T)^2}}
\]
where $A>0$ is fixed and $j=0,\ldots,J$  with $J$ the maximal integer such that $e^J/(\log\log T)^2\leqs 1/10^{12}$ (so that $J\asymp \log\log\log T$). 
In this section we take $A=1$ although later we need to take it sufficiently large. Let 
\[
\theta_j=\frac{e^j}{(\log\log T)^2},\,\,\,\,\,\ell_j=\theta_j^{-3/4}
\]
so that $T_j=T^{\theta_j}$ for $0\leqs j\leqs J$. 
Now, write
\[
w_j(p)=\frac{1}{p^{1/\theta_j\log T}}\frac{\log(T_j/p)}{\log T_j}
\]
and 
\[\mathcal{P}_{i,j}(t) := \sum_{T_{i-1} < p \leqs T_{i}} \frac{w_j(p)}{p^{1/2+it}}.\]
With this notation the first sum over primes in Lemma \ref{sound lem} with $x=T_j$ can be written as
\begin{equation}\label{sum decomp 1}
\sum_{p\leqs T_j}\frac{w_j(p)}{p^{{1}/{2}+it}}=\sum_{i=0}^{j} \mc{P}_{i,j}(t).
\end{equation}

Also, write 
\[
\mathcal{N}_{i,j}(t,k)=\sum_{\substack{\Omega(n)\leqs 10\ell_i\\ p|n\implies T_{i-1}< p\leqs T_i}}\frac{k^{\Omega(n)}w_j(n)\mf{g}(n)}{n^{1/2+it}}
\]
and note that on the set of $t\in [T,2T]$ such that $|k\mc{P}_{i,j}(t)|\leqs \ell_i$ we have 
\begin{equation}
\label{exp poly}
\exp(2k\Re\mc{P}_{i,j}(t))=(1+O(e^{-9\ell_i}))\big|\mathcal{N}_{i,j}(t,k)\big|^2
\end{equation}
by \eqref{exp trunc} and \eqref{trunc form}. Accordingly, if $t$ is such that $|k\mc{P}_{i,j}(t)|\leqs \ell_i$ for all $0\leqs i\leqs j$ then on applying \eqref{sum decomp 1} we have 
\begin{equation}\label{prod}
\exp\Big(2k\Re\sum_{p\leqs T_j}\frac{w_j(p)}{p^{{1}/{2}+it}}\Big)\ll\prod_{i=0}^j\big|\mathcal{N}_{i,j}(t,k)\big|^2
\end{equation}
since $\sum_{i=0}^je^{-9\ell_i}$ is a rapidly converging series. Note this is a Dirichlet polynomial of length
\begin{equation}\label{sum length}
\leqs \prod_{i=0}^JT_i^{10\ell_i}
= 
T^{10\sum_{i=0}^J\theta_j^{1/4}}
\leqs
T^{20e^{J/4}/(\log\log T)^{1/2}}\leqs T^{1/50}.
\end{equation}
We can now state an upper bound for the zeta function in terms of these short Dirichlet polynomials.

\begin{lem}\label{zeta lem} Assume RH. Then either 
\[
|k\mc{P}_{0,j}(t)|> \ell_0
\] 
for some $0\leqs j\leqs J$ or 
\begin{multline*}
|\zeta(\tfrac{1}{2}+it)|^{2k}
\ll
\bigg( \prod_{i=0}^J \big|\mathcal{N}_{i,J}(t,k)\big|^2
+
\sum_{\substack{ 0\leqs j\leqs J-1\\ j+1\leqs l\leqs J}} \exp\Big(\frac{2k}{\theta_j}\Big)
 \bigg(\frac{|k\mc{P}_{j+1,l}(t)|}{\ell_{j+1}}\bigg)^{2s_j}\prod_{i=0}^j \big|\mathcal{N}_{i,j}(t,k)\big|^2\bigg)
 \\
 \times |\mc{M}(t,k)|^2 
 \end{multline*}
for any positive integers $s_j$ where 
\[
\mc{M}(t,k) =\sum_{\substack{\Omega(n)\leqs 10k(\log_2 T)^2\\ p|n\implies  p\leqs \log T}}\frac{(k/2)^{\Omega(n)}\mf{g}(n)}{n^{1+2it}}.
\]
\end{lem}
\begin{proof}
Suppose $|k\mc{P}_{0,j}(t)|<\ell_0$. For $0\leqs j\leqs J-1$ let 
\[
S(j)=\left\{t\in[T,2T]: 
\begin{array}{lr}
&|k\mc{P}_{i,l}(t)|\leqs \ell_i\,\,\,\,\,\,\,\,\,\,\,\forall 1\leqs i\leqs j,\,\, \forall j\leqs l\leqs J;\\
&|k\mc{P}_{j+1,l}(t)|>\ell_{j+1} \text{ for some } j+1\leqs l\leqs J
\end{array}
\right\}
\]
and 
\[
S(J)=\bigg\{t\in[T,2T]: |k\mc{P}_{i,J}(t)|\leqs \ell_i\,\,\,\,\,\,\,\forall 1\leqs i\leqs J\bigg\}.
\]
Then since $[T,2T]=\cup_{j=0}^J S(j)$,  for $t\in[T,2T]$ we have
\begin{equation}\label{zeta decomp}
|\zeta(\tfrac{1}{2}+it)|^{2k}
%=\sum_{j=0}^J\mathds{1}_{t\in S(j)}\cdot|\zeta(\tfrac{1}{2}+it)|^{2k}
\leqs 
\mathds{1}_{t\in S(J)}\cdot|\zeta(\tfrac{1}{2}+it)|^{2k}+\sum_{\substack{0\leqs j\leqs J-1\\j+1\leqs l\leqs J}}\mathds{1}_{t\in S_l(j)}\cdot|\zeta(\tfrac{1}{2}+it)|^{2k}
\end{equation}
where 
\[
S_l(j)=\left\{t\in[T,2T]: 
\begin{array}{lr}
|k\mc{P}_{i,l}(t)|\leqs \ell_i &\,\,\,\,\,\,\,\,\,\,\,\forall 1\leqs i\leqs j,\,\, \forall j\leqs l\leqs J;\\
|k\mc{P}_{j+1,l}(t)|>\ell_{j+1} &
\end{array}
\right\}.
\]
%\[
%S_l(j)=\left\{t\in[T,2T]: 
%|k\mc{P}_{i,l}(t)|\leqs \ell_i\,\,\forall 1\leqs i\leqs j,\,\, \forall j\leqs l\leqs J;
%\,\,|k\mc{P}_{j+1,l}(t)|>\ell_{j+1}
%\right\}.
%\]

We apply Lemma \ref{sound lem} to each zeta function on the right hand side of \eqref{zeta decomp}. If $t\in S_l(j)$ then we take $x=T_j$ to give
\[
|\zeta(\tfrac{1}{2}+it)|^{2k}
\ll
\exp\bigg(2k\Re\sum_{p\leqs T_j}\frac{w_j(p)}{p^{1/2+it}}+2k\Re\sum_{p\leqs \log T}\frac{1}{2p^{1+2it}}
+\frac{2k}{\theta_j}\bigg).
\]
For the first sum over primes in the exponential we apply \eqref{prod}. For the second sum we note that, since $|\sum_{p\leqs \log T}\frac{1}{p^{1+2it}}|\leqs 2\log_3 T\leqs (\log\log T)^2$, we have 
\begin{equation*}
\begin{split}
\exp\Big(2k\Re\sum_{p\leqs \log T}\frac{1}{2p^{1+2it}}\Big)
= &
\big(1+O(e^{-9k(\log_2 T)^2})\big)
\bigg|\sum_{\substack{\Omega(n)\leqs 10k(\log\log T)^2\\ p|n\implies  p\leqs \log T}}\frac{(k/2)^{\Omega(n)}\mf{g}(n)}{n^{1+2it}}\bigg|^2
\end{split}
\end{equation*}
by \eqref{exp trunc} and \eqref{trunc form}. This is $\ll |\mc{M}(t,k)|^2$. Finally, to capture the small size of the set, we multiply by 
\[
 \bigg(\frac{|k\mc{P}_{j+1,l}(t)|}{\ell_{j+1}}\bigg)^{2s_j}>1.
 \] 
 If $t\in S(J)$ then we omit this last step since of course there is no such $P_{J+1,l}(t)$. 
\end{proof}

\subsection{Proof of the upper bound in Proposition \ref{Z prop}}

%By combining Lemma \ref{zeta lem} along with the observations of Section \ref{tools sec} we have enough to prove the upper bound in Proposition \ref{Z prop}. 

From \eqref{zeta D int} we are required to show that 
\[
\frac1T\int_T^{2T}|\zeta(\tfrac{1}{2}+it)|^{2k}|D(t,-k)|^2dt
\ll
\bigg(\frac{\log T}{\log X}\bigg)^{k^2}.
\]
To apply Lemma \ref{zeta lem} we must consider two cases; that where $|k\mc{P}_{0,j}(t)|> \ell_0$ and otherwise. We consider the former case first since this is simpler. 

So, let $E\subset [T,2T]$ be the subset on which $|k\mc{P}_{0,j}(t)|> \ell_0$, that is, when 
\[
\bigg|\sum_{p\leqs T^{1/(\log\log T)^2}}\frac{w_j(p)}{p^{1/2+it}}\bigg|>\frac{(\log\log T)^{3/2}}{k}.
\]
By Chebyshev's inequality and Lemma \ref{sound lem 2} we have 
\begin{align*}
\mu(E)
\leqs  &
\bigg(\frac{k^2}{(\log\log T)^{3}}\bigg)^m\int_T^{2T}\bigg|\sum_{p\leqs T^{1/(\log\log T)^2}}\frac{w_j(p)}{p^{1/2+it}}\bigg|^{2m}dt
\\
\ll &
Tm!\bigg(\frac{k^2}{(\log\log T)^{3}}\bigg)^m \bigg(\sum_{p\leqs T^{1/(\log\log T)^2}}\frac{1}{p}\bigg)^m
\end{align*}
provided $m\leqs (1-o(1))(\log\log T)^2$ where in the last line we have used $|w_j(p)|\leqs 1$ for all $j$. By Stirling's formula and Mertens' theorem this is  
\begin{equation}\label{meas E}
\mu(E)
\ll
Tm^{1/2}\bigg(\frac{k^2m}{e(\log\log T)^{2}}\bigg)^m
\leqs 
Te^{-c(\log\log T)^2}
\end{equation}
for some $c>1$ on choosing $m=\lfloor \min(1,\tfrac{1}{k^2})(\log\log T)^2\rfloor $.
Therefore by H\"older's inequality, Harper's bound for the moments of zeta and \eqref{D mth moment} we have 
\[
\frac{1}{T}\int_E|\zeta(\tfrac{1}{2}+it)|^{2k}|D(t,-k)|^2dt\ll e^{-C(\log\log T)^2}(\log T)^{O(1)}=o(1).
\]

We may now consider the second case where $|k\mc{P}_{0,j}(t)|> \ell_0$ and accordingly concentrate on the integral
\[
\frac{1}{T}\int_{[T,2T]\backslash E}|\zeta(\tfrac{1}{2}+it)|^{2k}|D(t,-k)|^2dt.
\]
By the second part of Lemma \ref{zeta lem} this is 
\begin{align}\label{upper bound}
\ll &
\frac{1}{T}\int_T^{2T}
\bigg( \prod_{i=0}^J \big|\mathcal{N}_{i,J}(t,k)\big|^2
+
\sum_{\substack{ 0\leqs j\leqs J-1\\ j+1\leqs l\leqs J}} \exp\Big(\frac{2k}{\theta_j}\Big)
 \bigg(\frac{|k\mc{P}_{j+1,l}(t)|}{\ell_{j+1}}\bigg)^{2s_j}\prod_{i=0}^j \big|\mathcal{N}_{i,j}(t,k)\big|^2\bigg)\nonumber
 \\
& \qquad\qquad\times |\mc{M}(t,k)|^2 |D(t,-k)|^2dt.
\end{align}
To compute the resultant integrals we apply the following lemma.

\begin{lem}\label{big int lem} 
For $0\leqs s_j\leqs 1/(10\theta_j)$ we have 
\begin{equation*}
\frac{1}{T}\int_T^{2T} |D(t,-k)|^2|\mc{M}(t,k)|^2|\mc{P}_{j+1,l}(t)|^{2s_j}\prod_{i=0}^j \big|\mathcal{N}_{i,j}(t,k)\big|^2 dt
\ll  s_j!P_{j+1}^{s_j}\bigg(\frac{\log T_j}{\log X}\bigg)^{k^2}
\end{equation*}
where 
\[
P_{j+1}=\sum_{T_j< p\leqs T_{j+1}}\frac{1}{p}.
\]
\end{lem}
\begin{proof}
We write the integrand as a multiple sum. 
First off, by the multinomial theorem we have 
\[
\mc{P}_{j+1,l}(t)^{s_j}=\bigg(\sum_{T_j< p\leqs T_{j+1}}\frac{w_l(p)}{p^{1/2+it}}\bigg)^{s_j}
=
s_j!\sum_{\substack{\Omega(n)=s_j\\p|n\implies T_j< p\leqs T_{j+1} }}
\frac{w_l(n)\mf{g}(n)}{n^{1/2+it}}.
\]
Therefore
\begin{align*}
&D(t,-k)\mc{M}(t,k) \mc{P}_{j+1,l}(t)^{s_j}\prod_{i=0}^j\mathcal{N}_{i,j}(t,k) 
\\
= &
s_j!\sum_{n\in S(X)}\frac{\alpha_{-k}(n)}{n^{1/2+it}}
\sum_{\substack{\Omega(n)\leqs 10k(\log_2 T)^2\\ p|n\implies  p\leqs \log T}}\frac{(k/2)^{\Omega(n)}\mf{g}(n)}{n^{1+2it}}
\sum_{\substack{\Omega(n)=s_j\\p|n\implies T_j< p\leqs  T_{j+1} }}
\frac{w_l(n)\mf{g}(n)}{n^{1/2+it}}
\\
&
\times 
\prod_{i=0}^j 
\bigg(\sum_{\substack{\Omega(n)\leqs 10\ell_i\\ p|n\implies T_{i-1}< p\leqs T_i}}\frac{k^{\Omega(n)}w_j(n)\mf{g}(n)}{n^{1/2+it}}\bigg).
\end{align*}
Since $X\leqs T_0$ we may group together all the sums over $T_0$-smooth numbers as a single sum and write the above as 
\[
s_j!\sum_{n}\frac{\gamma(n)}{n^{1/2+it}}
\sum_{\substack{\Omega(n)=s_j\\p|n\implies T_j<p\leqs  T_{j+1} }}
\frac{w_l(n)\mf{g}(n)}{n^{1/2+it}}
\prod_{i=1}^j 
\bigg(\sum_{\substack{\Omega(n)\leqs 10\ell_i\\ p|n\implies T_{i-1}< p\leqs T_i}}\frac{k^{\Omega(n)}w_j(n)\mf{g}(n)}{n^{1/2+it}}\bigg)
\]
for some coefficients $\gamma(n)$ where the product is empty if $j=0$. Since this is a short Dirichlet polynomial we find by \eqref{4.6} that 
\begin{multline}\label{big prod}
 \frac{1}{T}\int_T^{2T} |D(t,-k)|^2|\mc{M}(t,k)|^2|\mc{P}_{j+1,l}(t)|^{2s_j}\prod_{i=0}^j \big|\mathcal{N}_{i,j}(t,k)\big|^2 dt
\\
\ll 
{(s_j!)^2}\sum_{n}\frac{\gamma(n)^2}{n}
\sum_{\substack{\Omega(n)=s_j\\p|n\implies T_j\leqs p< T_{j+1} }}
\frac{w_l(n)^2\mf{g}(n)^2}{n} 
\\
\prod_{i=1}^j 
\bigg(\sum_{\substack{\Omega(n)\leqs 10\ell_i\\ p|n\implies T_{i-1}< p\leqs T_i}}\frac{k^{\Omega(n)}w_j(n)^2\mf{g}(n)^2}{n}\bigg).
\end{multline}

Now, by \eqref{prod mv} we explicitly have
\begin{equation*}
\sum_n \frac{\gamma(n)^2}{n}
=
\sideset{}{'}\sum_{\substack{n_1n_2n_3^2=\\n_4n_5n_6^2}}
\frac{\alpha_{-k}(n_1)\alpha_{-k}(n_4)k^{\Omega(n_2n_5)}(k/2)^{\Omega(n_3n_6)}W(\ul{n})}{(n_1n_2n_4n_5)^{1/2}n_3n_6}
\end{equation*}
where
\[
W(\ul{n})
=
w_j(n_2)w_j(n_5)\mf{g}(n_2)\mf{g}(n_3)\mf{g}(n_5)\mf{g}(n_6)
\]
and the $\prime$ in the sum denotes that $n_1,n_3\in S(X)$ and 
\[
\begin{array}{cc}
\Omega(n_2),\Omega(n_5)\leqs 10\ell_0&\qquad \Omega(n_3),\Omega(n_6)\leqs 10k(\log_2 T)^2\\
p|n_2,n_5\implies  p< T_0&\qquad  p|n_3,n_6\implies p\leqs \log T.
 \end{array}
\]
We first estimate the terms with $\Omega(n_1),\Omega(n_4)> W_0$. By the usual arguments, for $1<r<2$ these terms are bounded by, 
\begin{align*}
\ll &
r^{-W_0}\sum_{\substack{n_1n_2n_3^2=\\n_4n_5n_6^2\\p|n_j\implies p\leqs T}}
\frac{r^{\Omega(n_1)}d_k(n_1)d_k(n_4)k^{\Omega(n_2n_3n_5n_6)}W(\ul{n})}{(n_1n_2n_4n_5)^{1/2}n_3n_6}
\\
= &
r^{-W_0}\prod_{p\leqs T} \bigg(1+\frac{(2r+2)k^2}{p}+O(p^{-2})\bigg)\ll e^{-W_0}(\log T)^{6k^2}=o(1).
\end{align*}
We then replace $\alpha_{-k}(n)$ with $\beta_{-k}(n)$ in the remaining sum. The usual arguments also allow us to remove the restrictions on all the $\Omega(n_j)$ at the cost of an error of size $o(1)$. Expressing the resultant sum as an Euler product we get 
\begin{align*}
\sum_{n}\frac{\gamma(n)^2}{n}
=&
\prod_{p\leqs X}\bigg(1+O(p^{-2})\bigg)\prod_{X<p\leqs T_0}\bigg(1+\frac{k^2}{p}+O(p^{-2})\bigg) +o(1)
\\
\asymp &
 \bigg(\frac{\log T_0}{\log X}\bigg)^{k^2}
\end{align*}
since $\beta_{-k}(n)$ is supported on $X$-smooth numbers where it satisfies $\beta_{-k}(p)=-k$ and $\beta_{-k}(p^m)\ll d_k(p^m)$ for $m\geqs 2$.

The second sum in \eqref{big prod} is
\[
(s_j!)^2\sum_{\substack{\Omega(n)=s_j\\p|n\implies T_j\leqs p< T_{j+1} }}
\frac{w_l(n)^2\mf{g}(n)^2}{n}
\leqs
s_j!\sum_{\substack{\Omega(n)=s_j\\p|n\implies T_j\leqs p< T_{j+1} }}
\frac{s_j!\mf{g}(n)}{n}
=
s_j!\bigg(\sum_{T_j\leqs p<T_{j+1}}\frac{1}{p}\bigg)^{s_j},
\]
since $\mf{g}(n)\leqs 1$, whilst the third sum is 
\begin{align*}
\prod_{i=1}^j 
\sum_{\substack{\Omega(n)\leqs 10\ell_i\\ p|n\implies T_{i-1}< p\leqs T_i}}\frac{k^{2\Omega(n)}w_j(n)^2\mf{g}(n)^2}{n}
\ll &
\prod_{i=1}^j 
\prod_{T_{i-1}\leqs p<T_i}
\bigg(1+\frac{k^2}{p}+O(p^{-2})\bigg)
\\
\asymp &
\bigg(\frac{\log T_j}{\log T_0}\bigg)^{k^2}.
\end{align*}
Combining these estimates gives the result.
\end{proof}

\begin{proof}[Completion of proof of upper bound in Proposition \ref{Z prop}]
Applying Lemma \ref{big int lem} in \eqref{upper bound} gives an upper bound of the form 
\begin{align*}
\bigg(\frac{\log T_J}{\log X}\bigg)^{k^2}
+
\sum_{\substack{ 0\leqs j\leqs J-1\\ j+1\leqs l\leqs J}} \exp\Big(\frac{2k}{\theta_j}\Big)
s_j! \bigg(\frac{k{P}_{j+1}}{\ell_{j+1}^2}\bigg)^{s_j}\bigg(\frac{\log T_j}{\log X}\bigg)^{k^2}.
\end{align*}
On noting that 
\[
J-j\ll \log(1/\theta_j),\qquad P_{j+1}=\log\bigg(\frac{\log T_{j+1}}{\log T_j}\bigg)+o(1)\leqs 2, \qquad
\ell_{j+1}^2=\theta_j^{-3/2}
\]
and setting $s_j=1/(10\theta_j)$, then by Stirling's formula this is 
\begin{align*}
\ll &
\bigg(\frac{\log T}{\log X}\bigg)^{k^2}
\bigg(
1+\sum_{0\leqs j\leqs J-1}
\log(\tfrac{1}{\theta_j}) 
\theta_j^{-{1/10\theta_j}}c^{1/\theta_j} \theta_j^{{3}/{20\theta_j}}
\bigg)
\\
\ll &
 \bigg(\frac{\log T}{\log X}\bigg)^{k^2}
\bigg(
1+\sum_{0\leqs j\leqs J-1}e^{-{C}\theta_j^{-1}\log(1/\theta_j)}
\bigg)
\end{align*}
for some constants $c,C>0$. Since this last series is bounded the result follows.
\end{proof}

%%%%%%%%%%%%% LOWER 1 %%%%%%%%%%%%%%

\section{Lower bound for $0\leqs k\leqs 1$}\label{lower 1}

We keep a similar setup to the previous section but with a few minor changes. There is no longer any need for the weights $w_j(p)$ so we can simplify our notation and let 
\[
\mc{N}_i(t,k)
=
\sum_{\substack{\Omega(n)\leqs 10\ell_i\\ p|n\implies T_{i-1}< p\leqs T_i}}\frac{k^{\Omega(n)}\mf{g}(n)}{n^{1/2+it}}
\]
where the $T_i$ and $\ell_i$ are as before for $0\leqs i\leqs J$. As remarked earlier, in this section we take $T_{-1}=A $ with
$A$ sufficiently large to be chosen later.
We then form the product 
\begin{equation}\label{N prod}
\mc{N}(t,k):=\prod_{i=0}^J \mc{N}_i(t,k)=\sum_{n\leqs Y}\frac{\gamma_k(n)}{n^{1/2+it}}
\end{equation}
for some coefficients $\gamma_k(n)$ where from \eqref{sum length} we have $Y=T^{1/50}$. We can think of this as an approximation to $\zeta(1/2+it)^k$; It possesses several nice features akin to an Euler product whilst also being a short Dirichlet polynomial.

We acquire our lower bound by applying H\"older's inequality, the form of which will depend on whether $0< k\leqs 1$ or $k\geqs 1$. 
The latter case is somewhat simpler so we give details for the case $0< k\leqs 1$ first. By H\"older's inequality, we have  
\begin{align*}\label{holder}
&\Big|\frac{1}{T}\int_\mc{S}\zeta(\tfrac{1}{2}+it)\mc{N}(t,k-1)\overline{\mc{N}(t,k)}|D(t,-k)|^{2}dt
\Big| 
\\
&\ll
 \Big( \frac{1}{T}\int_\mc{S} |\zeta(\tfrac{1}{2}+it)|^{2k}|D(t,-k)|^2dt \Big)^{\frac 12} 
  \Big( \frac1T\int_{T}^{2T} |\zeta(\tfrac 12+it) {\mathcal N}(t,k-1)|^2
|D(t,-k)|^{2} dt \Big)^{\frac{1-k}{2}}
\\
&\hskip 1 in \times \Big(\frac1T\int_T^{2T}  |{\mathcal N}(t,k)|^{\frac{2}{k}} |{\mathcal N}(t,k-1)|^2 |D(t,-k)|^{2}dt \Big)^{\frac k2}. 
\end{align*}

Since $D(t,-k)\sim P_X(1/2+it)^{-k}$ for $t\in\mc{S}$ the first integral on the right hand side is $\ll \frac 1T\int_{T}^{2T}|Z_X(1/2+it)|^{2k}dt$. 
\begin{rem}
Note also that in this argument we can modify the definition of $D$ to be 
\[
D(t,k)=\sum_{p|n\implies A<p\leqs X}\frac{\alpha_k(n)}{n^{1/2+it}}
\]
since the removal of the $A$-smooth numbers from the sum in 
\[
\exp\Big(k\sum_{n\leqs X}\frac{\Lambda(n)}{n^{1/2+it}\log n}\Big)
\] 
leads to a bounded multiplicative factor which can be absorbed into the $\ll$ sign. Also, to save space in the future we may absorb the condition $p|n\implies A<p\leqs X$ into the coefficients $\alpha_k(n)$. 
\end{rem}
The lower bound in the case $0<k\leqs 1$ now follows from the subsequent Propositions.

\begin{prop}\label{lower prop}Suppose $X\leqs \eta_k(\log T)^2(\log_2 T)^2$. Then 
\[
\frac1T\int_\mc{S}\zeta(\tfrac{1}{2}+it)\mc{N}(t,k-1)\overline{\mc{N}(t,k)}|D(t,-k)|^{2}dt\geqs C\bigg(\frac{\log T}{\log X}\bigg)^{k^2}
\]
for some $C>0$. Assuming RH we may take $X\leqs (\log T)^{\theta_k-\epsilon}$.
\end{prop}

\begin{prop}\label{upper prop 1}For $X\leqs T^{1/(\log_2 T)^2}$ we have 
\[
\frac1T\int_{T}^{2T} |\zeta(\tfrac 12+it) {\mathcal N}(t,k-1)|^2  |D(t,-k)|^{2} dt \ll\bigg(\frac{\log T}{\log X}\bigg)^{k^2}.
\]
\end{prop}

\begin{prop}\label{upper prop 2} For $X\leqs T^{1/(\log_2 T)^2}$ we have 
\[
\frac1T\int_T^{2T}  |{\mathcal N}(t,k)|^{\frac{2}{k}} |{\mathcal N}(t,k-1)|^2|D(t,-k)|^{2}dt\ll\bigg(\frac{\log T}{\log X}\bigg)^{k^2}.
\]
\end{prop}

Note that our argument works unconditionally provided $X\leqs \eta_k(\log T)^2(\log_2 T)^2$ as claimed in the introduction. 

%%%%%%%%%%%%%% LOWER PROP 

\subsection{Proof of Proposition \ref{lower prop} }

Our first job is to extend the range of integration to the full set $[T,2T]$. 
By the usual argument involving H\"older's inequality the integral over $\mc{E}$ is $o(1)$. Indeed, from the conditions on $X$ and Lemma \ref{E meas lem} we have $T^{-1}\mu(\mc{E})\ll e^{-2|k|V_0}$ which is enough to kill any power of $\log T$. We also have the second moment bound for the zeta function, and by \eqref{D mth moment} the $m$th moment of $D(t,-k)$ is also $(\log T)^{O(1)}$. The only new ingredient required is a moment bound for $\mc{N}(t,k)$ but by the Montgomery--Vaughan mean value theorem this is, for $m\leqs 50$, 
\begin{align*}
\frac{1}{T}\int_{T}^{2T}|\mc{N}(t,k)|^{2m}dt
\ll &
\sum_{\substack{n_1\cdots n_m=n_{m+1}\cdots n_{2m}\\ n_j\leqs T^{1/50}}} \frac{\gamma_k(n_1)\cdots \gamma_k(n_{2m})}{(n_1\cdots n_{2m})^{1/2}}
\\
= &
\prod_{p\leqs T}\bigg(1+\frac{m^2k^2}{p}+O(p^{-2})\bigg)\ll (\log T)^{m^2k^2}.
\end{align*}
Therefore, 
\[
\frac1T\int_\mc{S}\zeta(\tfrac{1}{2}+it)\mc{N}(t,k-1)\overline{\mc{N}(t,k)}|D(t,-k)|^{2}dt
=
I_3+o(1)
\]
where 
\[
I_3
=
\frac1T\int_T^{2T}\zeta(\tfrac{1}{2}+it)\mc{N}(t,k-1)\overline{\mc{N}(t,k)}|D(t,-k)|^{2}dt.
\]

To lower bound $I_3$ we apply the approximation 
\[
\zeta(\tfrac{1}{2}+it)=\sum_{n\leqs T} \frac{1}{n^{1/2+it}}+O\Big(\frac{1}{T^{1/2}}\Big).
\]
The other terms of the integrand satisfy pointwise bounds of the form 
\[
\mc{N}(t,k)\ll Y^{1/2+\epsilon}\ll T^{1/100+\epsilon},\qquad \sum_{n} \frac{\alpha_{-k}(n)}{n^{1/2+it}}\ll X^{W_0(1/2+\epsilon)}\ll  T^{1/100};
\]
since $k^{\Omega(n)}$ has average order $(\log n)^{k-1}$ and $\alpha_{-k}(n)\ll d_k(n)\ll n^\epsilon$. We then see that the error term in the approximation for zeta leads to an error of size $o(1)$.

Therefore
\[
I_3
=
\frac1T\int_{T}^{2T} \sum_{\substack{n_1\leqs T,\,\, n_2,n_3\leqs Y}}\frac{\gamma_{k-1}(n_2)\gamma_k(n_3)\alpha_{-k}(n_4)\alpha_{-k}(n_5)}{(n_1n_2n_3n_4n_5)^{1/2}}\Big(\frac{n_3n_5}{n_1n_2n_4}\Big)^{it}dt+o(1).
\]
By direct integration, the off-diagonal terms for which $n_1n_2n_4\neq n_3n_5$ lead to an error of size
\[
\frac{YX^{W_0}}{T}\sum_{n_1\leqs T,\,\,n_2,n_3\leqs Y}\frac{|\gamma_{k-1}(n)|\gamma_k(n)}{(n_1n_2n_3)^{1/2}}\bigg(\sum_{n}\frac{|\alpha_{-k}(n)|}{n^{1/2}}\bigg)^2
\ll
\frac{Y^{2+\epsilon}X^{W_0(2+\epsilon)}}{T^{1/2}}=o(1)
\]
since $|\log(n_3n_5/n_1n_2n_4)|\gg 1/(YX^{W_0})$. Accordingly, 
\[
I_3
=
\sum_{\substack{n_1n_2n_4=n_3n_5\\n_1\leqs T,n_2,n_3\leqs Y}}\frac{\gamma_{k-1}(n_2)\gamma_k(n_3)\alpha_{-k}(n_4)\alpha_{-k}(n_5)}{(n_1n_2n_3n_4n_5)^{1/2}}+o(1).
\]
Since $n_3n_5\leqs YX^{W_0}\leqs T$ we may remove the condition $n_1\leqs T$. 

Now, since $\alpha_{-k}(n)$ is supported on prime powers $p^m$ with $A<p\leqs X$ we may write our sum as
\begin{align}\label{sum decomp}
\sum_{n_1n_2n_4=n_3n_5}
=
\sum_{\substack{n_1n_2n_4=n_3n_5\\p|n_j\implies A<p\leqs X}}
\cdot
\sum_{\substack{n_1n_2=n_3\\p|n_j\implies X<p\leqs T_0}}
\cdot
\prod_{i=1}^J\sum_{\substack{n_1n_2=n_3\\p|n_j\implies T_{i-1}<p\leqs T_i}} 
\end{align}
and note that we have essentially taken $n_4=n_5=1$ in the second two sums on the right. Unfolding the coefficients using \eqref{N prod} gives the first sum as 
\begin{equation}\label{first sum}
\sum_{\substack{n_1n_2n_4=n_3n_5\\p|n_j\implies A<p\leqs X\\\Omega(n_2),\Omega(n_3)\leqs 10\ell_0}}
\frac{(k-1)^{\Omega(n_2)}k^{\Omega(n_3)}\mf{g}(n_2)\mf{g}(n_3)\alpha_{-k}(n_4)\alpha_{-k}(n_5)}{(n_1n_2n_3n_4n_5)^{1/2}}
\end{equation}

The usual arguments now allow us to replace $\alpha_{-k}(n_j)$ with $\beta_{-k}(n_j)$ and remove the restrictions on $\Omega(n_2),\Omega(n_3)$ at the cost of a term of size $o(1)$. We then express the resultant sum as an Euler product. Since $\beta_{-k}(p)=-k$ we find that the leading terms cancel and that \eqref{first sum} is 
\begin{equation}\label{first sum bound}
\prod_{A< p\leqs X}\Big(1+O_k(p^{-2})\Big)+o(1).
\end{equation}
On taking $A$ sufficiently large we can guarantee that the term $O_k(p^{-2})$ is always $<1$ in modulus and hence the above product is $\geqs c$ for some constant $c>0$.

With a similar computation the second sum in \eqref{sum decomp} is 
\begin{equation}
\begin{split}
\label{second sum}
&\sum_{\substack{n_1n_2=n_3\\p|n_2,n_3\implies X<p\leqs T_0\\ \Omega(n_2),\Omega(n_3)\leqs 10\ell_0}}
\frac{(k-1)^{\Omega(n_2)}k^{\Omega(n_3)}\mf{g}(n_2)\mf{g}(n_3)}{(n_1n_2n_3)^{1/2}}
\\
= &
\sum_{\substack{n_1n_2=n_3\\p|n_2,n_3\implies X<p\leqs T_0}}
\frac{(k-1)^{\Omega(n_2)}k^{\Omega(n_3)}\mf{g}(n_2)\mf{g}(n_3)}{(n_1n_2n_3)^{1/2}}
+
O(e^{-10\ell_0}(\tfrac{\log T_0}{\log X})^{O(1)})
\\
= &
\prod_{X<p\leqs T_0}\bigg(1+\frac{k^2}{p}+O(p^{-2})\bigg)+o(1)
\geqs 
c\bigg(\frac{\log T_0}{\log X} \bigg)^{k^2}.
\end{split}
\end{equation}

For the sums inside the product in \eqref{sum decomp}, we must be a little more careful. The sums in question are given by  
\begin{equation}\label{third sum}
\sum_{\substack{n_1n_2=n_3\\p|n_2,n_3\implies T_{i-1}<p\leqs T_i\\ \Omega(n_2),\Omega(n_3)\leqs 10\ell_i}}
\frac{(k-1)^{\Omega(n_2)}k^{\Omega(n_3)}\mf{g}(n_2)\mf{g}(n_3)}{(n_1n_2n_3)^{1/2}}.
\end{equation}
Since $0<k\leqs 1$ and $d(n)\leqs 2^{\Omega(n)}\leqs e^{\Omega(n)}$, the error incurred from dropping the condition on $\Omega(n_2)$ is, in absolute value,
\begin{align*}
\leqs &
e^{-10\ell_i}\sum_{\substack{n_1n_2=n_3\\p|n_2,n_3\implies T_{i-1}<p\leqs T_i}}
\frac{e^{\Omega(n_2)}\mf{g}(n_2)\mf{g}(n_3)}{(n_1n_2n_3)^{1/2}}
\leqs
e^{-10\ell_i}\sum_{\substack{p|n\implies T_{i-1}<p\leqs T_i}}
\frac{e^{2\Omega(n)}\mf{g}(n)}{n}
\\
= &
\exp\Big(-10\ell_i+e^2\sum_{T_{i-1}<p\leqs T_i}\frac{1}{p}\Big)
\leqs 
\exp\Big(-10\ell_i+e^2\log(\tfrac{\log T_i}{\log T_{i-1}})+o(1)\Big)
\leqs 
e^{-9\ell_i}
\end{align*}
since $\log(\tfrac{\log T_i}{\log T_{i-1}})\leqs 2$ and $\ell_i\geqs 10^9$. Doing the same for $n_3$ gives an error of the same size and hence the sum in \eqref{third sum} is 
\begin{equation}
\begin{split}
\label{third sum 2}
\geqs &
\prod_{T_{i-1}<p\leqs T_i}\bigg(1+\frac{k^2}{p}+O(p^{-2})\bigg)
-
e^{-8\ell_i}
\\
\geqs &
\big(1-e^{-7\ell_i}\big)\prod_{T_{i-1}<p\leqs T_i}\bigg(1+\frac{k^2}{p}+O(p^{-2})\bigg).
\end{split}
\end{equation}
Hence, combining \eqref{first sum bound}, \eqref{second sum} and \eqref{third sum 2} we find that 
\[
I_3
\geqs 
c \bigg(\frac{\log T_0}{\log X}\bigg)^{k^2}\prod_{i=1}^J \big(1-e^{-7\ell_i}\big)\prod_{T_{i-1}<p\leqs T_i}\bigg(1+\frac{k^2}{p}+O(p^{-2})\bigg)
\geqs 
C  \bigg(\frac{\log T}{\log X}\bigg)^{k^2}
\]
since, again, $\sum_{i=1}^J e^{-7\ell_i}$ is a rapidly converging series. This completes the proof of Proposition \ref{lower prop}.

%%%%%%%%%%%%%%%%% PROP 7

\subsection{Proof of Proposition \ref{upper prop 1}}

We are required to show
\[
I_4
:=
\frac1T\int_{T}^{2T} |\zeta(\tfrac 12+it) {\mathcal N}(t,k-1)|^2|D(t,-k)|^{2} dt \ll\bigg(\frac{\log T}{\log X}\bigg)^{k^2}.
\]
From the conditions on $X$ given in the statement of the proposition, $D(t,-k)$ is a Dirichlet polynomial of length $X^{W_0}=T^{o(1)}$ which is still short. Thus, we can apply the results of \cite{BCH,BCR} after combining the two Dirichlet polynomials into a single sum. This gives 
\[
I_2=\sum_{m,n}\frac{h_k(m)h_k(n)(m,n)}{mn} \log \Big( \frac{B T (m,n)^2}{mn} \Big)+o(1)
\] 
for some constant $B$ where 
\[
h_k(n)=\sum_{n_1n_2=n}\gamma_{k-1}(n_1)\alpha_{-k}(n_2).
\]
As in \cite{HS}, we write 
$$ 
\log \Big(  \frac{B T (m,n)^2}{mn} \Big) = \frac{1}{2\pi i } \int_{|z|= 1/\log T} \Big( \frac{BT (m,n)^2}{mn}\Big)^z \frac{dz}{z^2}, 
$$
so that the main term in $I_2$ becomes
$$ 
\frac{1}{2\pi i } \int_{|z|= 1/\log T}\sum_{m, n} \frac{h_k(m)h_k(n)(m,n)}{mn}  \Big( \frac{BT(m,n)^2}{mn}\Big)^z  \frac{dz}{z^2}. 
$$ 
Then after a trivial estimate we get
\begin{equation}\label{I_2 bound}
I_4 \ll (\log T)\max_{|z|=1/\log T}\bigg|\sum_{m, n} \frac{h_k(m)h_k(n)(m,n)^{1+2z}}{(mn)^{1+z}}   \bigg|
\end{equation}
and thus we are required to show this last sum is $\ll \tfrac{1}{\log X}(\log T/\log X)^{k^2-1}$.

Unfolding the coefficients, the  above sum becomes
\begin{align*}
\sum_{m_1,m_2,n_1,n_2}\frac{\gamma_{k-1}(m_1)\gamma_{k-1}(n_1)\alpha_{-k}(m_2)\alpha_{-k}(n_2)(m_1m_2,n_1n_2)^{1+2z}}{(m_1m_2n_1n_2)^{1+z}}. 
\end{align*}
Estimating the terms with $\Omega(m_2),\Omega(n_2)\geqs W_0$ in the usual way we may replace $\alpha_{-k}(n)$ by $\beta_{-k}(n)$ and then re-extend the sums at the cost of $o(1)$.
Then by multiplicativity we can express the resultant sum as 
\[
\sum_{\substack{m_1,m_2,n_1,n_2\\p|m_jn_j\implies A<p\leqs X}}
\cdot
\sum_{\substack{m_1,n_1\\p|m_1n_1\implies X< p\leqs T_0}}
\cdot
\prod_{i=1}^J
\sum_{\substack{m_1,n_1\\p|m_1n_1\implies T_{i-1}< p\leqs T_i}}
\]
where, again, we have taken $m_2=n_2=1$ in the second two sums since the functions $\beta_k(n)$ are supported on $X$-smooth numbers. As usual, we drop the conditions on $\Omega(m_j),\Omega(n_j)$ in these sums. Considering, for the moment, the first sum without these conditions we get 
\begin{align*}
&
\prod_{A<p\leqs X}\bigg(1+\sum_{\substack{0\leqs m_j,n_j\leqs 1\\m_1+m_2+n_1+n_2\neq 0}}\frac{(k-1)^{m_1+n_1}(-k)^{m_2+n_2}(p^{m_1+m_2},p^{n_1+n_2})^{1+2z}}{p^{(m_1+m_2+n_1+n_2)(1+z)}}+O(p^{-2})\bigg)
\\
= &
\prod_{A<p\leqs X}\bigg(1+\frac{2(k-1)-2k}{p^{1+z}}+\frac{(k-1)^2+k^2-2k(k-1)}{p}+O(p^{-2})\bigg)
\\
= &
\prod_{A<p\leqs X}\bigg(1-\frac{1}{p}+O(p^{-2})+O\Big(\frac{1}{\log T}\frac{\log p}{p}\Big)\bigg)
\ll 
\frac{1}{\log X}
\end{align*}
where we have used $\beta_{-k}(n)\ll d_k(n)$ and $|z|\leqs 1/\log T$ along with the bound $\frac{1}{\log T}\sum_{p\leqs X}(\log p)/p\ll 1$.
As usual, for the error term we apply Rankin's trick along with similar Euler product computations to give an error of size 
\[
\ll e^{-10\ell_0+O(\log\log X)}
\]
which is $o(1/\log X)$. 

For the second sum we get, with a similar argument, 
\begin{align*}
\sum_{\substack{m_1,n_1\\p|m_1n_1\implies X< p\leqs T_0}}
&
\frac{\gamma_{k-1}(m_1)\gamma_{k-1}(n_1)(m_1,n_1)^{1+2z}}{(m_1n_1)^{1+z}} 
\\
= &
\prod_{X<p\leqs T_0}\bigg(1+\frac{k^2-1}{p}+O(p^{-2})\bigg)+O(e^{-10\ell_0+O(\log(\log T_0/\log X))} )
%\\
%\ll &
%\bigg(\frac{\log T_0}{\log X}\bigg)^{k^2-1}
\end{align*}
which is $\ll (\frac{\log T_0}{\log X})^{k^2-1}$. Finally, for the product of sums we have 
\begin{align*}
\prod_{i=1}^J
&
\sum_{\substack{m_1,n_1\\p|m_1n_1\implies T_{i-1}< p\leqs T_i}}
\frac{\gamma_{k-1}(m_1)\gamma_{k-1}(n_1)(m_1,n_1)^{1+2z}}{(m_1n_1)^{1+z}} 
\\
= &
\prod_{i=1}^J
\bigg[
\prod_{T_{i-1}<p\leqs T_i}\bigg(1+\frac{k^2-1}{p}+O(p^{-2})\bigg)+O(e^{-10\ell_i+O(\log(\log T_i/\log T_{i-1}))} )
\bigg]
\\
= &
\prod_{i=1}^J
\bigg[\big(1+O(e^{-10\ell_i} )\big)
\prod_{T_{i-1}<p\leqs T_i}\bigg(1+\frac{k^2-1}{p}+O(p^{-2})\bigg)
\bigg]
\end{align*}
since $\log T_i/\log T_{i-1}\leqs 2$. This last term is $\ll (\frac{\log T}{\log T_0})^{k^2-1}$ and so combining these bounds in \eqref{I_2 bound} we get 
\[
I_4 \ll \log T\cdot\frac{1}{\log X}\cdot \bigg(\frac{\log T_0}{\log X}\bigg)^{k^2-1}\cdot\bigg(\frac{\log T}{\log T_0}\bigg)^{k^2-1}\ll \bigg(\frac{\log T}{\log X}\bigg)^{k^2}
\]
which completes the proof of Proposition \ref{upper prop 1}.

%%%%%%%%%%%%%%%%% PROP 8

\subsection{Proof of Proposition \ref{upper prop 2} }\label{prop 9 subsec} We begin with the following lemma from \cite{HS}. 

\begin{lem} \label{lem10} Let
\[
\mathcal{P}_{j}(t) := \sum_{T_{j-1} < p \leqs T_{j}} \frac{1}{p^{1/2+it}}.
\]
Then for $0\leqs j\leqs J$ 
\[ 
|{\mathcal N}_j(t, k-1) {\mathcal N}_j(t, k)^{\frac{1}{k}}|^2 \leqs 
|{\mathcal N}_j(t, k)|^2 (1+ O(e^{-9\ell_j})) + O\big( {\mathcal Q}_j(t)\big), 
\] 
where the implied constants are absolute, and 
\[ 
{\mathcal Q}_j(t) =\Big( \frac{e |{\mathcal P}_j(t)|}{10\ell_j} \Big)^{20\ell_j} \sum_{r=0}^{10\ell_j/k} \Big( \frac{2e |{\mathcal P}_j(t)|}{r+1} \Big)^{2r}.
\]

\end{lem} 
\begin{proof}
This is essentially Lemma 1 of \cite{HS}. Our sequence $T_j$ is defined slightly differently but one can check that this makes no difference to the end result. 
\end{proof}

\begin{lem} 
\label{lem11} With the above notation
\[
\frac{1}{T}\int_T^{2T} {\mathcal Q}_0(t) |D(t,-k)|^{2}dt \ll  e^{-10\ell_0}(\log X)^{2k^2}
\] 
and for $1\leqs j\leqs J$
\[
\frac{1}{T}\int_T^{2T} {\mathcal Q}_j(t) dt \ll  e^{-10\ell_j}. 
\] 
\end{lem} 

\begin{proof}
We prove the first bound since this is new, the second bound follows similarly (and is essentially Lemma 2 of \cite{HS}). Let $L=10\ell_0$. From the definition of $\mc{Q}_0(t)$ we have 
\begin{align*}
\frac{1}{T}\int_T^{2T} & {\mathcal Q}_0(t) |D(t,-k)|^{2}dt 
\\
= &
\Big( \frac{12}{L} \Big)^{2L} \sum_{r=0}^{L/k} \Big( \frac{2e}{r+1} \Big)^{2r}
\cdot
\frac{1}{T}\int_T^{2T}|{\mathcal P}_0(t)|^{2L+2r}|D(t,-k)|^{2}dt
\end{align*}
By the Cauchy--Schwarz inequality and \eqref{D fourth} the integral is  
\[
\ll (\log X)^{2k^2} \bigg(\frac{1}{T}\int_T^{2T}|{\mathcal P}_0(t)^{2L+2r}|^2dt\bigg)^{1/2}.
\]
Since 
\[
{\mathcal P}_0(t)^{2L+2r}
=
(2L+2r)!
\sum_{\substack{\Omega(n)=2L+2r\\p|n\implies p\leqs T_0}}\frac{\mf{g}(n)}{n^{1/2+it}},
\]
the Montgomery--Vaughan mean value theorem gives that our original integral is
\begin{equation}
\begin{split}
\label{Q int bound}
\ll &
(\log X)^{2k^2}
\Big( \frac{12}{L} \Big)^{2L} \sum_{r=0}^{L/k} \Big( \frac{2e}{r+1} \Big)^{2r}
\cdot
\bigg((2L+2r)!^2 \sum_{\substack{\Omega(n)=2L+2r\\p|n\implies p\leqs T_0}}\frac{\mf{g}(n)^2}{n}\bigg)^{1/2}
\\
\leqs &
(\log X)^{2k^2}
\Big( \frac{12}{L} \Big)^{2L} \sum_{r=0}^{L/k} \Big( \frac{2e}{r+1} \Big)^{2r}
\cdot
(2L+2r)!^{1/2}\bigg(\sum_{p\leqs T_0}\frac{1}{p}\bigg)^{L+r}
\end{split}
\end{equation}
since $\mf{g}(n)^2\leqs \mf{g}(n)$. Letting $P=\sum_{p\leqs T_0}p^{-1}$ we find by Stirling's formula that the summand is 
\[
\ll 
(2/e)^{L}(8e)^{r} (r+1)^{-2r} ({L+r})^{L+r+{1}/{4}}P^{L+r}
\]
which is maximised at the solution of $r^2=8P(L+r)(1+O(1/r))$. Since $P\leqs 2\log\log T=o(L)$ this solution $r=r_0$ satisfies 
$2\sqrt{2}(PL)^{1/2}\leqs r_0\leqs 3(PL)^{1/2}$. Therefore, \eqref{Q int bound} is 
\[
\ll 
(\log X)^{2k^2}
\Big( \frac{c}{L} \Big)^{2L}\cdot \frac{L}{k} \cdot
(3L)!^{1/2}P^{L+r_0}r_0^{-2r_0}
\]
since $2r_0\leqs L$. This is then 
\[
\ll 
(\log X)^{2k^2}L^{-L/2+o(1)}
\ll 
e^{-10\ell_0}
\]
and the result follows.
\end{proof}

By Lemma \ref{lem10} and \eqref{4.6} we find that 
\begin{align*}
I_5
:= &
\frac1T\int_T^{2T}  |{\mathcal N}(t,k)|^{\frac{2}{k}} |{\mathcal N}(t,k-1)|^2 |D(t,-k)|^{2}dt
\\
\ll &
\frac1T\int_T^{2T}\Big(|{\mathcal N}_0(t, k)|^2 (1+ O(e^{-10\ell_0})) + O\big( {\mathcal Q}_0(t)\big)\Big) |D(t,-k)|^{2}dt
\\
& \times
\prod_{j=1}^{J+1} 
\frac1T\int_T^{2T}\Big(|{\mathcal N}_j(t, k)|^2 (1+ O(e^{-10\ell_j})) + O\big( {\mathcal Q}_j(t)\big)\Big)dt
\end{align*}
since $J\asymp \log\log \log T$.
By Lemma \ref{lem11} we have 
\begin{equation}\label{first prod}
\begin{split}
&
\prod_{j=1}^{J} 
\frac1T\int_T^{2T}\Big(|{\mathcal N}_j(t, k)|^2 (1+ O(e^{-10\ell_j})) + O\big( {\mathcal Q}_j(t)\big)\Big)dt
\\
= &
\prod_{j=1}^{J} 
\bigg( (1+ O(e^{-10\ell_j}))\sum_{\substack{\Omega(n)\leqs 10\ell_j\\ p|n\implies T_{i-1}<p\leqs T_i}}\frac{k^{2\Omega(n)}\mf{g}(n)^2}{n}+O(e^{-10\ell_j})\bigg)
\\
\ll &
\prod_{j=1}^{J} \prod_{T_{i-1}<p\leqs T_i}\bigg(1+\frac{k^2}{p}+O(p^{-2})\bigg)
\ll 
\bigg(\frac{\log T}{\log T_0}\bigg)^{k^2}.
\end{split}
\end{equation}
By Lemma \ref{lem11} again we find 
\begin{align*}
&
\frac1T\int_T^{2T}\Big(|{\mathcal N}_0(t, k)|^2 (1+ O(e^{-10\ell_0})) + O\big( {\mathcal Q}_0(t)\big)\Big) |D(t,-k)|^{2}dt
\\ = &
(1+ O(e^{-10\ell_0})) \frac1T\int_T^{2T}|{\mathcal N}_0(t, k)|^2 \bigg|\sum_n \frac{\alpha_{-k}(n)}{n^{{1}/{2}+it}}\bigg|^{2}dt +o(1).
\end{align*}
This last integral is 
\[
\sum_{\substack{m_1n_1=m_2n_2\\\Omega(m_j)\leqs 10\ell_0\\ p|m_j\implies A<p\leqs T_0}}\frac{k^{\Omega(m_1)+\Omega(m_2)}\mf{g}(m_1)\mf{g}(m_2)\alpha_{-k}(n_1)\alpha_{-k}(n_2)}{(m_1m_2n_1n_2)^{1/2}}+o(1).
\]
The usual arguments allow us to remove the conditions on $\Omega(m_j)$ and replace $\alpha_{-k}(n)$ by $\beta_{-k}(n)$ at the cost of  $o(1)$. We then find that the resultant sum is 
\[
\prod_{A<p\leqs X}\Big(1+O(p^{-2})\Big)\prod_{X<p\leqs T_0}\Big(1+\frac{k^2}{p}+O(p^{-2})\Big)
\ll
\bigg(\frac{\log T_0}{\log X}\bigg)^{k^2} 
\]
since the leading terms cancel over $A<p\leqs X$. 
Combining this with \eqref{first prod} gives
\[
I_5\ll \bigg(\frac{\log T}{\log X}\bigg)^{k^2} 
\] 
thus completing the proof of Proposition \ref{upper prop 2}.

%%%%%%%%%%%%%%% LOWERR 2 %%%%%%%%%%%

\section{The lower bound of Proposition \ref{Z prop} for $k\geqs 1$} \label{lower 2}

The lower bound for $k\geqs 1$ is similar to the case $0\leqs k\leqs 1$, if not a little simpler. In this case we take 
\[T_{-1}=Bk^2\]
for some $B>0$ to be chosen 
and we alter the definition of $J$ slightly so that it is the maximal integer such that $e^J/(\log\log T)^2\leqs 1/(10^{12}k^4)$. This implies that
$\ell_j=\theta_j^{-3/4}\geqs 10^9k^3$ for all $j\leqs J$.

We perform H\"older's inequality in the form 
\begin{multline}
\label{holder 2}
\Big|\frac{1}{T}\int_\mc{S}\zeta(\tfrac{1}{2}+it)\mc{N}(t,k-1)\overline{\mc{N}(t,k)}|D(t,-k)|^{2}dt
\Big| 
\\
\ll \Big( \frac{1}{T}\int_T^{2T} |Z_X(\tfrac 12+it)|^{2k} dt \Big)^{\frac{1}{2k}} \Big( \frac{1}{T}\int_T^{2T} |{\mathcal N}(t, k-1) {\mathcal N}(t, k)|^{\frac{2k}{2k-1}} |D(t,-k)|^2dt \Big)^{\frac{2k-1}{2k}} 
\end{multline}
where again we have used Lemma \ref{P lem}. The integral on the left can be dealt with rather similarly to Proposition \ref{lower prop} although the change in parameters requires some modifications. We detail these alterations first before dealing with the second integral on the right.

\subsection{Modifying the proof of Proposition \ref{lower prop}}

We see that we can arrive at \eqref{sum decomp} in exactly the same way. When dealing with \eqref{first sum}, the errors incurred from dropping the conditions on $\Omega(n_2)$ etc. $\!\!$ are now 
\[
\ll e^{-10\ell_0}(\log X)^{Ck^2}
\] 
which of course is still $o(1)$ since $k$ is fixed. Note that on writing \eqref{first sum} as a single sum $\sum_n f(n)/n$, the coefficients are supported on $X$-smooth numbers and satisfy the bounds $|f(n)|\leqs k^{2\Omega(n)}d_3(n)^2\leqs (3k)^{2\Omega(n)}$. Then we find that the equivalent of \eqref{first sum} is 
\[
\prod_{Bk^2<p\leqs X}\bigg(1+O\Big(\frac{k^4}{p^2}\Big)\bigg)+o(1)\geqs C
\]
for some $C$ on taking $B$ large enough. In a similar way we find that the equivalent of \eqref{second sum} is 
\[
\prod_{X<p\leqs T_0}\bigg(1+\frac{k^2}{p}+O\Big(\frac{k^4}{p^2}\Big)\bigg)+o(1)\geqs C\bigg(\frac{\log T_0}{\log X}\bigg)^{k^2}.
\]

It remains to deal with the term 
\begin{equation}\label{third sum 3}
\prod_{i=1}^J\sum_{\substack{n_1n_2=n_3\\p|n_2,n_3\implies T_{i-1}<p\leqs T_i\\ \Omega(n_2),\Omega(n_3)\leqs 10\ell_i}}
\frac{(k-1)^{\Omega(n_2)}k^{\Omega(n_3)}\mf{g}(n_2)\mf{g}(n_3)}{(n_1n_2n_3)^{1/2}}
\end{equation}
in the current context. The error from removing the condition on $\Omega(n_2)$ in the sums is 
\begin{align*}
\leqs &
e^{-10\ell_i}\sum_{\substack{n_1n_2=n_3\\p|n_2,n_3\implies T_{i-1}<p\leqs T_i}}
\frac{(ek)^{\Omega(n_2)}k^{\Omega(n_3)}\mf{g}(n_2)\mf{g}(n_3)}{(n_1n_2n_3)^{1/2}}
\\
\leqs &
e^{-10\ell_i}\sum_{\substack{p|n\implies T_{i-1}<p\leqs T_i}}
\frac{(ek)^{2\Omega(n)}\mf{g}(n)}{n}
= 
\exp\Big(-10\ell_i+e^2k^2\sum_{T_{i-1}<p\leqs T_i}\frac{1}{p}\Big)
\\
\leqs &
\exp\Big(-10\ell_i+e^2k^2\log(\tfrac{\log T_i}{\log T_{i-1}})+o(1)\Big)
\leqs 
e^{-9\ell_i}
\end{align*}
where we have used $d(n)\leqs 2^{\Omega(n)}\leqs e^{\Omega(n)}$. Then \eqref{third sum 3} is 
\begin{equation*}
\begin{split}
\label{third sum 4}
\geqs &
\prod_{i=1}^J\bigg(\prod_{T_{i-1}<p\leqs T_i}\bigg(1+\frac{k^2}{p}+O\Big(\frac{k^4}{p^{2}}\Big)\bigg)
-
e^{-8\ell_i}\bigg)
\\
\geqs &
\prod_{i=1}^J\big(1-e^{-7\ell_i}\big)\prod_{T_{i-1}<p\leqs T_i}\bigg(1+\frac{k^2}{p}+O\Big(\frac{k^4}{p^{2}}\Big)\bigg)
\geqs 
C\bigg(\frac{\log T}{\log T_0}\bigg)^{k^2}.
\end{split}
\end{equation*}
Combining these we get the desired bound
\begin{equation}\label{I int lower}
\frac{1}{T}\int_\mc{S}\zeta(\tfrac{1}{2}+it)\mc{N}(t,k-1)\overline{\mc{N}(t,k)}|D(t,-k)|^{2}dt\geqs C\bigg(\frac{\log T}{\log X}\bigg)^{k^2}.
\end{equation}

\subsection{The remaining integral}
We let 
\[
I_6
=
\frac{1}{T}\int_T^{2T} |{\mathcal N}(t, k-1) {\mathcal N}(t, k)|^{\frac{2k}{2k-1}} |D(t,-k)|^2dt.
\]
Then Proposition \ref{Z prop} in the case $k\geqs 1$ will follow from H\"older's inequality \eqref{holder 2} and \eqref{I int lower} if we can show that 
\begin{equation*}\label{I_4 bound}
I_6
\ll
\bigg(\frac{\log T}{\log X}\bigg)^{k^2}. 
\end{equation*}

\begin{lem}\label{N frac lem 2}
For $0\leqs i\leqs J$ we have 
\[
|{\mathcal N_i}(t, k-1) {\mathcal N_i}(t, k)|^{\frac{2k}{2k-1}}
\leqs 
(1+O(e^{-9\ell_i}))|\mc{N}_i(t,k)|^2+\Big(\frac{ek|\mc{P}_i(t)|}{10\ell_i}\Big)^{40\ell_i}.
\]
\end{lem}
\begin{proof}
%This follows similarly to the proof of  Lemma 1 of \cite{HS}. 
If $k|\mc{P}_i(t)|\leqs 10\ell_i$ then by \eqref{exp trunc}  we have 
\begin{align*}
|{\mathcal N_i}(t, k-1) {\mathcal N_i}(t, k)|^{\frac{2k}{2k-1}}
= &
(1+O(e^{-9\ell_i}))\big|\exp\big((k-1)\mc{P}_i(t)+k\mc{P}_i(t)\big)\big|^{\frac{2k}{2k-1}}
\\
= &
(1+O(e^{-9\ell_i}))\big|\exp\big(k\mc{P}_i(t)\big)\big|^{2}
\\
= &
(1+O(e^{-9\ell_i}))|\mc{N}_i(t,k)|^2.
\end{align*}
If $k|\mc{P}_i(t)|> 10\ell_i$ then 
\begin{align*}
|{\mathcal N_i}(t, k-1)|
\leqs &
\sum_{r=0}^{10\ell_i}\frac{((k-1)|\mc{P}_i(t)|)^r}{r!}
\leqs (k|\mc{P}_i(t)|)^{10\ell_i}\sum_{r=0}^{10\ell_i}(10\ell_i)^{r-10\ell_i}\frac{1}{r!}
\\
\leqs &
\Big(\frac{ek|\mc{P}_i(t)|}{10\ell_i}\Big)^{10\ell_i}.
\end{align*}
The same bound holds for $|\mc{N}_i(t,k)|$ and hence the result follows since $2k/(2k-1)\leqs 2$.
\end{proof}

\begin{lem} 
\label{lem11 2} We have 
\[
\frac{1}{T}\int_T^{2T}\Big(\frac{ek|\mc{P}_0(t)|}{10\ell_0}\Big)^{40\ell_0} |D(t,-k)|^{2}dt \ll  e^{-10\ell_0}(\log X)^{2k^2}
\] 
and for $1\leqs i\leqs J$. 
\[
\frac{1}{T}\int_T^{2T}  \Big(\frac{ek|\mc{P}_i(t)|}{10\ell_i}\Big)^{40\ell_i}dt \ll  e^{-10\ell_i}. 
\] 
\end{lem} 
\begin{proof}
We prove the first formula since the second follows similarly. By the Cauchy--Schwarz inequality and \eqref{D fourth} we have 
\[
\frac{1}{T}\int_T^{2T}|\mc{P}_0(t)|^{40\ell_0} |D(t,-k)|^{2}dt
\ll
 (\log X)^{2k^2}\bigg(\frac{1}{T}\int_T^{2T}|\mc{P}_0(t)^{40\ell_0}|^2dt\bigg)^{1/2}.
\]
Letting $L=10\ell_0=10(\log\log T)^{3/2}$ the last integral is 
\[
\ll 
(4L)!^2\sum_{\substack{\Omega(n)=4L\\p|n\implies T_{-1}<p\leqs T_0}}\frac{\mf{g}(n)^2}{n}
\leqs 
(4L)!\bigg(\sum_{T_{-1}<p\leqs T_0}\frac{1}{p}\bigg)^{4L}
\leqs 
(4L)!(\log\log T)^{4L}
\]
using $\mf{g}(n)^2\leqs \mf{g}(n)$.
Thus, the original integral is 
\[
\ll 
(\log X)^{2k^2}\bigg(\frac{ek}{L}\bigg)^{4L}(4L)!^{1/2}\Big(\frac{L}{10}\Big)^{4L/3}
\ll
(\log X)^{2k^2}L^{-2L/3}c^{L}
\]
by Stirling's formula. The result follows.
\end{proof}

As in subsection \ref{prop 9 subsec} we now combine Lemmas \ref{N frac lem 2} and \ref{lem11 2} along with \eqref{4.6} and the usual calculations to give 
\[
I_6\ll \bigg(\frac{\log T}{\log X}\bigg)^{k^2}.
\]
This completes the proof of Proposition \ref{Z prop} in the case $k\geqs 1$.

 %%%%%%%%%%%%%%%%%%     BIB    %%%%%%%%%

\end{document}